\documentclass[11pt]{article}

\usepackage{amsmath,amssymb}
\usepackage{latexsym}
\usepackage{graphicx}
\usepackage{bm}

\newtheorem{thm}{Theorem}

\newtheorem{lem}[thm]{Lemma}
\newtheorem{rem}[thm]{Remark}

\newenvironment{proof}{\begin{trivlist}
                       \item[]{\bf Proof.}
                       \hspace{0cm}}{\hfill $\Box$
                       \end{trivlist}}

\begin{document}
\title{Dynamical Systems Gradient method for solving\\ nonlinear equations
with monotone operators}

\author{N. S. Hoang$\dag$\footnotemark[1]
 \quad 	  A. G. Ramm$\dag$\footnotemark[3] \\
\\
\\
$\dag$Mathematics Department, Kansas State University,\\
Manhattan, KS 66506-2602, USA
}

\renewcommand{\thefootnote}{\fnsymbol{footnote}}
\footnotetext[1]{Email: nguyenhs@math.ksu.edu}
\footnotetext[3]{Corresponding author. Email: ramm@math.ksu.edu}

\date{}
\maketitle

\begin{abstract}
\noindent
A version of the Dynamical Systems Gradient Method for solving ill-posed
nonlinear monotone operator equations is studied in this paper.
A  discrepancy principle is proposed and justified.
A numerical experiment was carried out with the new stopping rule.
Numerical experiments show that the proposed stopping rule is efficient.
Equations with monotone operators are of interest in many applications.
\\


\end{abstract}

\section{Introduction}

In this paper we study a version of the Dynamical Systems Method (DSM) 
(see \cite{R499}) for solving the equation
\begin{equation}
\label{aeq1}
F(u)=f,
\end{equation}
 where
$F$ is a nonlinear, twice Fr\'{e}chet differentiable, monotone operator in a 
real Hilbert space $H$, and equation \eqref{aeq1} is assumed solvable, 
possibly nonuniquely. 
Monotonicity means that 
\begin{equation}
\label{ceq2}
\langle F(u)-F(v),u-v\rangle\ge 0,\quad \forall u,v\in H.
\end{equation}
Equations with monotone operators are important in many applications
and were studied extensively, see, for example, 
\cite{L}, \cite{P}, \cite{S},  \cite{V},  and references 
therein. One encounters many technical and 
physical problems with such 
operators in the cases where dissipation of energy occurs. For 
example, in \cite{R129} and \cite{R118}, Chapter 3, pp.156-189,
a wide class of nonlinear dissipative systems is studied, and the basic 
equations of such systems can be reduced to equation \eqref{aeq1}
with monotone operators.
Numerous examples of equations with monotone operators can be found in 
\cite{L} and references mentioned above. In \cite{R504} and \cite{R522}
it is proved that any solvable linear operator equation with a closed, 
densely defined operator in a Hilbert space $H$ can be reduced to an equation 
with a monotone operator and solved by a convergent iterative process.

In this paper, apparently for the 
first time, the 
convergence of the Dynamical Systems Gradient method is proved
under natural assumptions and convergence of a corresponding iterative 
method is established. 
No special assumptions of smallness of the nonlinearity or other special 
properties of the nonlinearity are imposed. No source-type 
assumptions are used. Consequently, our
result is quite general and widely applicable. It is well known, that
without extra assumptions, usually, source-type assumption about the
right-hand side, or some assumption concerning the  smoothness of the 
solution, one cannot get a specific rate of convergence
even for linear ill-posed equations (see, for example, \cite{R499}, 
where one can find a proof of this statement).
On the other hand, such assumptions are often difficult to verify
and often they do not hold. By this reason we do not make such 
assumptions.

The  result of this paper  is useful both because of its many possible 
applications and because of its general nature. Our novel technique 
consists of an application of some new inequalities.  Our main results are
formulated in Theorems \ref{mainthm} and \ref{mainthm2}, and also in several lemmas,
for example, in Lemmas \ref{lemma3}, \ref{lemma4}, \ref{lemma8}, \ref{lemma9}, \ref{lemma11}, 
\ref{lemma12}. Lemmas \ref{lemma3}, \ref{lemma4}, \ref{lemma11}, \ref{lemma12}
may be useful in many other problems.  

In \cite{Tau1} a stationary equation $F(u)=f$  with
a nonlinear  monotone operator
$F$ was studied.  The assumptions A1-A3 on p.197 in  
\cite{Tau1} 
are more restrictive than ours, and the Rule R2 on p.199,
formula (4.1) in  \cite{Tau1} for the choice of the regularization 
parameter
is quite different from our rule and is more difficult to use
it computationally: one has to solve a nonlinear equation 
(equation (4.1) in \cite{Tau1}) in order to find 
the regularization parameter. To use this equation one has to invert 
an ill-conditioned linear operator $A+\alpha I$ for small values of 
$\alpha$.
Assumption A1 in \cite{Tau1}
is not verifiable practically, because the solution $x^\dagger$
is not known. Assumption A3 in \cite{Tau1}  requires $F$ to be constant 
in 
a ball $B_r(x^\dagger)$ if $F'(x^\dagger)=0$. 
Our method does not require these assumptions, 
and, in contrast to equation (4.1) in   \cite{Tau1}, 
it does not require inversion of
ill-conditioned linear operators and solving nonlinear equations for 
finding the regularization parameter. The stopping time is
chosen numerically in our method without extra computational effort
by a discrepancy-type principle formulated and  justified in Theorem \ref{mainthm},
in Section \ref{mainsec}. We give a convergent iterative process for 
stable solution of equation (1.1) and a stopping rule for this process.
 
In \cite{Tau1} 
the "source-type assumption" is made, that is, it is assumed 
that the right-hand side of the equation $F(u)=f$ belongs to the range 
of a suitable operator. This usually allows one to get some convergence 
rate. 
In our paper, as was already mentioned above, such an assumption is not 
used because, on the one hand,
numerically it is difficult to verify such an assumption,
and, on the other hand, such an assumption may be not satisfied in many 
cases, even in linear ill-posed problems, for example, in the case when 
the solution does not have extra smoothness. 

We assume the nonlinearity to be twice locally  Fr\'{e}chet 
differentiable. This assumption, as we mention below, does not restrict
the global growth of the nonlinearity. In many practical and 
theoretical problems the nonlinearities are smooth and given
analytically.
In these cases one can calculate $F'$
analytically. This is the case in the example, considered in Section \ref{numsec}. 
This example is a simple model problem for non-linear Wiener-type 
filtering (see \cite{R486}). If one drops the nonlinear cubic term
in the equation $Bu+u^3=f$  of this example, then the resulting equation 
$Bu=f$ does not have integrable solutions, in general, even for 
very smooth $f$, for example, for $f\in 
C^\infty([0,1])$, as shown in \cite{R486}. It is, therefore, of special
interest to solve this equation numerically. 

It is known (see, e.g., \cite{R499}), that the set 
$\mathcal{N}:=\{u:F(u)=f\}$ is closed and convex if $F$ is monotone and 
continuous. 
A closed and convex set in a Hilbert space has a unique minimal-norm element. 
This element in $\mathcal{N}$ we denote by $y$, $F(y)=f$. 
We assume that 
\begin{equation}
\label{ceq3}
\sup_{\|u-u_0\|\le R}\|F^{(j)}(u)\|\le M_j(R),\quad 0\le j\le 2,
\end{equation}
where $u_0\in H$ is an element of $H$, $R>0$ is arbitrary, 
and $f=F(y)$ is not known but
 $f_\delta$, the noisy data, are known, and $\|f_\delta-f\|\le \delta$.
Assumption \eqref{ceq3} simplifies our arguments and
does not restrict the global growth of the nonlinearity.
In \cite{R454} this assumption is weakened to hemicontinuity 
in the problems related to the existence of the global solutions
of the equations, generated by the  DSM. In many applications the 
nonlinearity $F$ is given analytically, and then one can calculate  
$F'(u)$ analytically. 
  
If $F'(u)$ is not boundedly invertible then 
 solving equation \eqref{aeq1} for $u$ given noisy data $f_\delta$ is often (but not always) an ill-posed problem.
When $F$ is a linear bounded operator many methods for stable solving of \eqref{aeq1} were proposed (see \cite{R526}, \cite{I}--\cite{R499} and references therein). 
However, when $F$ is nonlinear then the theory is
less complete.

DSM consists of finding a nonlinear map $\Phi(t,u)$ such that the Cauchy problem
$$
\dot{u}=\Phi(t,u),\qquad u(0)=u_0,
$$
has a unique solution for all $t\ge0$, there exists $\lim_{t\to\infty}u(t):=u(\infty)$,
and $F(u(\infty))=f$,
\begin{equation}
\label{23eq1}
\exists !\, \,u(t)\quad \forall t\ge 0;\qquad \exists u(\infty);\qquad F(u(\infty))=f.
\end{equation}
Various choices of $\Phi$ were proposed in \cite{R499} for \eqref{23eq1} to hold.
Each such choice yields a version of the DSM.

The DSM for solving equation \eqref{aeq1} was extensively studied in 
\cite{R499}--\cite{491}. 
In \cite{R499}, the following version of the DSM was investigated
for monotone operators $F$:
\begin{equation}
\label{hichehic}
\dot{u}_\delta = -\big{(}F'(u_\delta) + a(t)I\big{)}^{-1}\big{(}F(u_\delta)+a(t)u_\delta - f_\delta\big{)},
\quad u_\delta(0)=u_0.
\end{equation}
Here $I$ denotes the identity operator in $H$. 
The convergence of this method was justified with some {\it a priori} 
choice of stopping rule.

In \cite{Tau} a continuous gradient method for solving equation \eqref{aeq1}
was studied. A stopping rule of discrepancy type was introduced and justified under the assumption 
that $F$ satisfies the following condition:
\begin{equation}
\label{taueq1}
\|F(\tilde{x}) - F(x) - F'(x)(\tilde{x} - x) \|= \eta \|F(x) - F(\tilde{x})\|, \qquad \eta < 1,
\end{equation}
for all $x,\, \tilde{x}$ in some ball  $B(x_0,R)\subset H$. This very restrictive assumption
is not satisfied even for monotone operators. Indeed, if $F'(x)=0$ for some
$x\in B(x_0)$ then \eqref{taueq1} implies $F(x)=f$ for all $x\in 
B(x_0,R)$, 
provided that
$B(x_0,R)$ contains a solution of \eqref{aeq1}. 

In this paper we consider a gradient-type version of the DSM for solving equation \eqref{aeq1}:
\begin{equation}
\label{aeq2}
\dot{u}_\delta = -\big{(}F'(u_\delta)^* + a(t)I\big{)}\big{(}F(u_\delta)+a(t)u_\delta - f_\delta\big{)},
\quad u_\delta(0)=u_0,
\end{equation}
where $F$ is a monotone operator and $A^*$ denotes the adjoint to a linear operator $A$.   
If $F$ is monotone then $F'(\cdot):=A\ge0$. If a bounded linear operator $A$
is defined on all of the complex Hilbert space $H$ and $A\ge0$,
i.e., $\langle Au,u\rangle \ge0,\, \forall u\in H$, then $A=A^*$,
so $A$ is selfadjoint. In a real Hilbert space $H$ a bounded linear operator defined on all of $H$
and satisfying the inequality $\langle Au,u\rangle \ge0, \, \forall u\in H$ is not necessary selfadjoint.
Example: $H=\mathbb{R}^2,\, A= \begin{pmatrix}2&1\\0&2 \end{pmatrix}$, 
$\langle Au,u\rangle = 2u_1^2 + u_1u_2 +u_2^2\ge 0$, but 
$A^*= \begin{pmatrix}2&0\\1&2 \end{pmatrix}\not=A$.

The convergence of the method \eqref{aeq2} for any initial value $u_0$ is proved
for a stopping rule based on a discrepancy principle. 
This {\it a posteriori} choice of stopping time 
$t_\delta$ is justified
provided that $a(t)$ is suitably chosen.

The advantage of method \eqref{aeq2}, a modified version of the gradient method, over the Gauss-Newton method and
the version \eqref{hichehic} of the DSM is the following: no inversion of matrices is needed in \eqref{aeq2}.
Although the convergence rate of the DSM \eqref{aeq2} maybe slower than that of the DSM 
\eqref{hichehic}, the DSM \eqref{aeq2} might be faster than the DSM 
\eqref{hichehic} for large-scale systems due to its lower computation cost at each iteration.

In this paper we investigate a stopping rule based on a discrepancy principle (DP) for the DSM \eqref{aeq2}. 
The main results of this paper are Theorem~\ref{mainthm} and Theorem~\ref{mainthm2} in which
a DP is formulated,  
the existence of a stopping time $t_\delta$ is proved, and
the convergence of the DSM with the proposed DP is justified under some natural assumptions.

\section{Auxiliary results}

The inner product in $H$ is denoted $\langle u,v\rangle$. Let us consider the following equation
\begin{equation}
\label{2eq2}
F(V_\delta)+aV_\delta-f_\delta = 0,\qquad a>0,
\end{equation}
where $a=const$. It is known (see, e.g., \cite{R499}, \cite{Z}) that equation \eqref{2eq2} 
with monotone continuous operator $F$ has a unique solution for any $f_\delta\in H$. 

Let us recall the following result from \cite{R499}:
\begin{lem}
\label{rejectedlem}
Assume that equation \eqref{aeq1} is solvable, $y$ is its minimal-norm solution, assumptions
\eqref{ceq2} holds, and $F$ is continuous. Then
$$
\lim_{a\to 0} \|V_{a}-y\| = 0,
$$
where $V_a$ solves \eqref{2eq2} with $\delta=0$.
\end{lem}
Of course, under our assumption \eqref{ceq3}, $F$ is continuous.

\begin{lem}
\label{lem0}
If \eqref{ceq2} holds and $F$ is continuous, then
$\|V_\delta\|=O(\frac{1}{a})$ as $a\to\infty$, and
\begin{equation}
\label{4eq2}
\lim_{a\to\infty}\|F(V_\delta)-f_\delta\|=\|F(0)-f_\delta\|.
\end{equation}
\end{lem}

\begin{proof}
Rewrite \eqref{2eq2} as
$$
F(V_\delta) - F(0) + aV_\delta + F(0)-f_\delta = 0.
$$
Multiply this equation
by $V_\delta$, use inequality $\langle F(V_\delta)-F(0),V_\delta-0\rangle \ge 0$ and get:
$$
a\|V_\delta\|^2\le
\|f_\delta-F(0)\|\|V_\delta\|.
$$
Therefore, 
$\|V_\delta\|=O(\frac{1}{a})$. This and the continuity of $F$ imply \eqref{4eq2}.
\end{proof}

Let $a=a(t)$ be strictly monotonically decaying continuous
positive function on $[0,\infty)$, $0<a(t)\searrow 0$, and assume $a\in C^1[0,\infty)$. 
These assumptions hold throughout the paper and often are not repeated. 
Then the solution $V_\delta$ of \eqref{2eq2} is a function of $t$, $V_\delta=V_\delta(t)$.
From the triangle inequality one gets:
$$
\|F(V_\delta(0))-f_\delta\|\ge\|F(0) - f_\delta\| -\|F(V_\delta(0))-F(0)\|.
$$
From Lemma~\ref{lem0} it follows that for large $a(0)$ one has:
$$
\|F(V_\delta(0))-F(0)\|\le M_1\|V_\delta(0)\|=O\bigg{(}\frac{1}{a(0)}\bigg{)}.
$$
Therefore, 
if $\|F(0)-f_\delta\|>C\delta$, then $\|F(V_\delta(0))-f_\delta\|\ge (C-\epsilon)\delta$, 
where $\epsilon>0$ is sufficiently small and $a(0)>0$ is sufficiently large. 

Below the words decreasing and increasing mean strictly decreasing and strictly increasing.

\begin{lem}
\label{lemma3}
Assume $\|F(0)-f_\delta\|>0$.
Let $0<a(t)\searrow 0$, and $F$ be monotone.
Denote 
$$
\psi(t) :=\|V_\delta(t)\|,\qquad \phi(t):=a(t)\psi(t)=\|F(V_\delta(t)) - f_\delta\|,
$$ 
where $V_\delta(t)$ solves \eqref{2eq2} with $a=a(t)$. 
Then
$\phi(t)$ is decreasing, and $\psi(t)$ is increasing.
\end{lem}
 
\begin{proof}
Since $\|F(0)-f_\delta \|>0$, one has $\psi(t)\not=0,\, \forall t\ge 0$. 
Indeed, if $\psi(t)\big{|}_{t=\tau}=0$, then $V_\delta(\tau)=0$, and equation \eqref{2eq2} implies
$\|F(0)-f_\delta\|=0$, which is a contradiction. 
Note that $\phi(t)=a(t)\|V_\delta(t)\|$. One has
\begin{equation}
\label{1eq3}
\begin{split}
0&\le \langle F(V_\delta(t_1))-F(V_\delta(t_2)),V_\delta(t_1)-V_\delta(t_2)\rangle\\
&= \langle -a(t_1)V_\delta(t_1)+a(t_2)V_\delta(t_2),V_\delta(t_1)-V_\delta(t_2)\rangle\\
&= (a(t_1)+a(t_2))\langle V_\delta(t_1),V_\delta(t_2) \rangle -a(t_1)\|V_\delta(t_1)\|^2 - a(t_2)\|V_\delta(t_2)\|^2.
\end{split}
\end{equation}
Thus,
\begin{equation}
\label{2eq6}
\begin{split}
0& \le  (a(t_1)+a(t_2))\|V_\delta(t_1)\|\|V_\delta(t_2) \| - a(t_1)\|V_\delta(t_1)\|^2 - a(t_2)\|V_\delta(t_2)\|^2\\
& = (a(t_1) \|V_\delta(t_1)\| - a(t_2) \|V_\delta(t_2)\|)(\|V_\delta(t_2)\|-\|V_\delta(t_1)\|)\\
& = (\phi(t_1)-\phi(t_2))(\psi(t_2) - \psi(t_1)).
\end{split}
\end{equation}

If $\psi(t_2) > \psi(t_1)$ then \eqref{2eq6} implies $\phi(t_1)\ge \phi(t_2)$, so
$$
a(t_1)\psi(t_1)\ge a(t_2)\psi(t_2)> a(t_2)\psi(t_1).
$$
Thus, if $\psi(t_2)> \psi(t_1)$ then $a(t_2)< a(t_1)$ and, therefore, $t_2> t_1$,
because $a(t)$ is strictly decreasing. 

Similarly, if $\psi(t_2)< \psi(t_1)$ then $\phi(t_1)\le \phi(t_2)$. This implies $a(t_2)> a(t_1)$, so $t_2< t_1$.

Suppose $\psi(t_1)=\psi(t_2)$, i.e., $\|V_\delta(t_1)\|=\|V_\delta(t_2)\|$. From \eqref{1eq3}, one has
$$
\|V_\delta(t_1)\|^2\le \langle V_\delta(t_1),V_\delta(t_2) \rangle \le \|V_\delta(t_1)\|\|V_\delta(t_2)\| = \|V_\delta(t_1)\|^2.
$$
This implies $V_\delta(t_1)=V_\delta(t_2)$, and then equation \eqref{2eq2} implies $a(t_1)=a(t_2)$. 
Hence, $t_1=t_2$, because $a(t)$ is strictly decreasing.

Therefore $\phi(t)$ is decreasing
and $\psi(t)$ is increasing.
\end{proof}

\begin{lem}
\label{lemma4}
Suppose that $\|F(0)-f_\delta\|> C\delta$, \,$C>1$, and $a(0)$ is sufficiently large. 
Then, there exists a unique $t_1>0$ such that $\|F(V_\delta(t_1))-f_\delta\|=C\delta$.
\end{lem}

\begin{proof}
The uniqueness of $t_1$ follows from Lemma~\ref{lemma3} because
$\|F(V_\delta(t))-f_\delta\|=\phi(t)$, and $\phi$ is decreasing. 
We have $F(y)=f$, and
\begin{align*}
0&=\langle F(V_\delta)+aV_\delta-f_\delta, F(V_\delta)-f_\delta \rangle\\
&=\|F(V_\delta)-f_\delta\|^2+a\langle V_\delta-y, F(V_\delta)-f_\delta \rangle + a\langle y, F(V_\delta)-f_\delta \rangle\\
&=\|F(V_\delta)-f_\delta\|^2+a\langle V_\delta-y, F(V_\delta)-F(y) \rangle + a\langle V_\delta-y, f-f_\delta \rangle 
+ a\langle y, F(V_\delta)-f_\delta \rangle\\
&\ge\|F(V_\delta)-f_\delta\|^2 + a\langle V_\delta-y, f-f_\delta \rangle + a\langle y, F(V_\delta)-f_\delta \rangle.
\end{align*}
Here the inequality $\langle V_\delta-y, F(V_\delta)-F(y) \rangle\ge0$ was used. 
Therefore
\begin{equation}
\label{1eq1}
\begin{split}
\|F(V_\delta)-f_\delta\|^2 &\le -a\langle V_\delta-y, f-f_\delta \rangle - a\langle y, F(V_\delta)-f_\delta \rangle\\
&\le a\|V_\delta-y\| \|f-f_\delta\| + a\|y\| \|F(V_\delta)-f_\delta\|\\
&\le  a\delta \|V_\delta-y\|  + a\|y\| \|F(V_\delta)-f_\delta\|.
\end{split}
\end{equation}
On the other hand, we have
\begin{align*}
0&= \langle F(V_\delta)-F(y) + aV_\delta +f -f_\delta, V_\delta-y\rangle\\
&=\langle F(V_\delta)-F(y),V_\delta-y\rangle + a\| V_\delta-y\| ^2 + a\langle y, V_\delta-y\rangle + \langle f-f_\delta, V_\delta-y\rangle\\
&\ge  a\| V_\delta-y\| ^2 + a\langle y, V_\delta-y\rangle + \langle f-f_\delta, V_\delta-y\rangle,
\end{align*}
where the inequality $\langle V_\delta-y, F(V_\delta)-F(y) \rangle\ge0$ was used. Therefore,
$$
a\|V_\delta-y\|^2 \le a\|y\|\|V_\delta-y\|+\delta\|V_\delta-y\|.
$$
This implies
\begin{equation}
\label{1eq2}
a\|V_\delta-y\|\le a\|y\|+\delta.
\end{equation}
From \eqref{1eq1} and \eqref{1eq2}, and an elementary inequality $ab\le \epsilon a^2+\frac{b^2}{4\epsilon},\,\forall\epsilon>0$, one gets:
\begin{equation}
\label{3eq4}
\begin{split}
\|F(V_\delta)-f_\delta\|^2&\le \delta^2 + a\|y\|\delta + a\|y\| \|F(V_\delta)-f_\delta\|\\
&\le \delta^2 + a\|y\|\delta + \epsilon \|F(V_\delta)-f_\delta\|^2 + 
\frac{1}{4\epsilon}a^2\|y\|^2,
\end{split}
\end{equation}
where $\epsilon>0$ is fixed, independent of $t$, and can be chosen 
arbitrary small. 
Let $t\to\infty$ and $a=a(t)\searrow 0$. Then \eqref{3eq4} implies
$$
\overline{\lim}_{t\to\infty}(1-\epsilon)\|F(V_\delta)-f_\delta\|^2\le 
\delta^2.
$$
This, the continuity of $F$, the continuity of $V_\delta(t)$ on $[0,\infty)$, and the assumption $\|F(0)-f_\delta\|>C\delta$ 
imply that equation $\|F(V_\delta(t))-f_\delta\|=C\delta$ must have a solution $t_1>0$.
The uniqueness of this solution has already established. 
\end{proof}

\begin{rem}
\label{rem2.9}
{\rm From the proof of Lemma~\ref{lemma4} one obtains the following claim: 

If $t_n\nearrow \infty$ then there exists a unique $n_1>0$ such that
$$
\|F(V_{n_1+1})-f_\delta\| \le C\delta < \|F(V_{n_1})-f_\delta\|,\qquad V_n:=V_\delta(t_n).
$$

}
\end{rem}

\begin{rem}
\label{remmoi}
{\rm
From Lemma~\ref{lem0} and Lemma~\ref{lemma3} one concludes that
$$
a_n\|V_n\|=\|F(V_n)-f_\delta\| \le \|F(0)-f_\delta\|,\qquad a_n:=a(t_n),\quad\forall n\ge 0.
$$
}
\end{rem}

\begin{rem}
\label{rem3}
{\rm 
Let $V:=V_\delta(t)|_{\delta=0}$, so 
$$
F(V)+a(t)V-f=0.
$$ 
Let $y$ be the minimal-norm solution to equation \eqref{aeq1}. 
We claim that
\begin{equation}
\label{rejected11}
\|V_{\delta}-V\|\le \frac{\delta}{a}.
\end{equation}
Indeed, from \eqref{2eq2} one gets
$$
F(V_{\delta}) - F(V) + a (V_{\delta}-V)=f- f_\delta.
$$
Multiply this equality with $(V_{\delta}-V)$ and use the monotonicity of $F$ to get
\begin{align*}
a \|V_{\delta}-V\|^2\le \delta \|V_{\delta}-V\|.
\end{align*}
This implies \eqref{rejected11}. 
Similarly, multiplying the equation
$$
F(V) + a V -F(y)=0,
$$
by $V-y$ one derives the inequality:
\begin{equation}
\label{rejected12}
\|V\| \le \|y\|.
\end{equation}
Similar arguments one can find in \cite{R499}. 

From \eqref{rejected11} and \eqref{rejected12}, one gets the following estimate:
\begin{equation}
\label{2eq1}
\|V_\delta\|\le \|V\|+\frac{\delta}{a}\le \|y\|+\frac{\delta}{a}.
\end{equation}
}
\end{rem}

\begin{lem}
\label{lemma8}
Suppose $a(t) = \frac{d}{(c+t)^b}$, $\varphi(t)=\int_0^t \frac{a^2(s)}{2}ds$
where $b\in (0,\frac{1}{4}]$, $d$ and $c$ are positive constants. Then 
\begin{equation}
\label{auxi1}
\frac{d^2}{2}\bigg{(}1-\frac{2b}{c^\theta d^2} \bigg{)}\int_0^t \frac{e^{\varphi(s)}}{(s+c)^{3b}} ds 
< \frac{e^{\varphi(t)}}{(c+t)^b},\qquad \forall t>0,\quad \theta=1-2b>0.
\end{equation}
\end{lem}

\begin{proof}
We have 
\begin{equation}
\label{zxeq14}
\varphi(t)= \int_0^t \frac{d^2}{2(c+s)^{2b}}ds
 = \frac{d^2}{2(1-2b)} \bigg{(}(c+t)^{1-2b} -c^{1-2b}\bigg{)} = p(c+t)^\theta - C_3,
\end{equation}
where $\theta:=1-2b,\, p:=\frac{d^2}{2\theta},\, C_3:= pc^\theta$.
One has
\begin{align*}
\frac{d}{dt}\frac{e^{p(c+t)^\theta}}{(c+t)^b} 
&= \frac{p\theta e^{p(c+t)^\theta}}{(c+t)^{b+1-\theta}}
- \frac{be^{p(c+t)^\theta}}{(c+t)^{b+1}}\\
&= \frac{e^{p(c+t)^\theta}}{(c+t)^b}\bigg{(}\frac{d^2}{2(c+t)^{2b}} - \frac{b}{c+t} \bigg{)}\\
&\ge \frac{e^{p(c+t)^\theta}}{(c+t)^b}\frac{d^2}{2(c+t)^{2b}}\bigg{(}1-\frac{2b}{c^\theta d^2} \bigg{)}.
\end{align*}
Therefore, 
\begin{align*}
\frac{d^2}{2} \bigg{(}1-\frac{2b}{c^\theta d^2} \bigg{)}\int_0^t \frac{e^{p(c+s)^\theta}}{(s+c)^{3b}} ds 
&\le \int_0^t\frac{d}{ds}\frac{e^{p(c+s)^\theta}}{(c+s)^b}ds\\
&\le \frac{e^{p(c+t)^\theta}}{(c+t)^b} - \frac{e^{pc^\theta}}{c^b}\le \frac{e^{p(c+t)^\theta}}{(c+t)^b}.
\end{align*}
Multiplying this inequality by $e^{-C_3}$ and using \eqref{zxeq14}, one obtains \eqref{auxi1}. 
Lemma~\ref{lemma8} is proved.
\end{proof}

\begin{lem}
\label{lemma9}
Let $a(t)=\frac{d}{(c+t)^b}$ and $\varphi(t):=\int_0^t \frac{a^2(s)}{2}ds$ 
where $d,c>0$, $b\in(0,\frac{1}{4}]$ and
$c^{1-2b} d^2\ge 6b$. One has
\begin{equation}
\label{auxieq3}
e^{-\varphi(t)}\int_0^t e^{\varphi(s)}|\dot{a}(s)|\|V_\delta(s)\|ds \le
 \frac{1}{2}a(t)\|V_\delta(t)\|,\qquad t\ge 0.
\end{equation} 
\end{lem}

\begin{proof} 
From Lemma \ref{lemma8}, one has
\begin{equation}
\label{yeq541}
\frac{1}{2}\bigg{(}1-\frac{2b}{c^\theta d^2} \bigg{)}\int_0^t e^{\varphi(s)}\frac{d^3}{(s+c)^{3b}} ds < e^{\varphi(t)}\frac{d}{(c+t)^b},\qquad \forall c,b\ge 0,\quad \theta=1-2b>0.
\end{equation}
Since $c^{1-2b}d^2\ge 6b$ or $\frac{6b}{c^\theta c^2_1}\le 1$, one has
$$
1-\frac{2b}{c^\theta d^2} \ge \frac{4b}{c^\theta d^2} \ge \frac{4b}{(c+s)^{1-2b}d^2},\qquad  s\ge 0.
$$
This implies 
\begin{equation}
\label{yeq55}
\frac{a^3(s)}{2}\bigg{(}1-\frac{2b}{c^\theta d^2}\bigg{)}
= \frac{d^3}{2(c+s)^{3b}}\bigg{(}1-\frac{2b}{c^\theta d^2}\bigg{)} \ge \frac{4db}{2(c+s)^{b+1}}=2|\dot{a}(s)|,\qquad s\ge 0.
\end{equation}
Multiplying \eqref{yeq541} by $\|V_\delta(t)\|$, using inequality \eqref{yeq55} and the fact that $\|V_\delta(t)\|$
is increasing, one gets, for all $t > 0$, the following inequalities:
$$
 e^{\varphi(t)}a(t)\|V_\delta(t)\| > \int_0^t e^{\varphi(s)}\|V_\delta(t)\|
\frac{a^3(s)}{2}\bigg{(}1-\frac{2b}{c^\theta d^2}\bigg{)} ds
 \ge 2\int_0^t e^{\varphi(s)} |\dot{a}(s)| \|V_\delta(s)\|  ds.
$$
This implies inequality \eqref{auxieq3}. Lemma~\ref{lemma9} is proved.
\end{proof}


Let us recall the following lemma, which is basic in our proofs. 
\begin{lem}[\cite{R499}, p. 97]
\label{lemramm}
Let $\alpha(t)$, $\beta(t)$, $\gamma(t)$ be continuous 
nonnegative functions on $[t_0,\infty)$, $t_0\ge 0$
is a fixed number. If there exists a function
$$
\mu\in C^1[t_0,\infty),\quad \mu>0, \quad \lim_{t\to\infty} \mu(t)=\infty,
$$
such that
\begin{align}
\label{1eq4}
0\le \alpha(t)&\le \frac{\mu}{2}\bigg{[}\gamma -\frac{\dot{\mu}(t)}{\mu(t)}\bigg{]},
\qquad \dot{\mu}:=\frac{d\mu}{dt},\\
\label{1eq5}
\beta(t)      &\le \frac{1}{2\mu}\bigg{[}\gamma -\frac{\dot{\mu}(t)}{\mu(t)}\bigg{]},\\
\label{1eq6}
\mu(0)g(0)    &< 1,
\end{align}
and $g(t)\ge 0$ satisfies the inequality
\begin{equation}
\label{1eq7}
\dot{g}(t)\le -\gamma(t)g(t)+\alpha(t)g^2(t)+\beta(t),\quad t\ge t_0,
\end{equation}
then $g(t)$ exists on $[t_0,\infty)$ and
\begin{equation}
\label{3eq10}
0\le g(t) < \frac{1}{\mu(t)}\to 0,\quad \text{as} \quad t\to\infty.
\end{equation}
If inequalities \eqref{1eq4}--\eqref{1eq6} hold on an interval $[t_0,T)$, then
$g(t)$ exists on this interval and inequality \eqref{3eq10} holds on $[t_0,T)$.
\end{lem}

\begin{lem}
\label{lemma11}
Suppose $M_1, c_0$, and $c_1$ are positive constants and $0\not=y\in H$.
Then there exist $\lambda>0$ and a function $a(t)\in C^1[0,\infty)$, $0<a(t)\searrow 0$, such that 
$$
|\dot{a}(t)|\le \frac{a^3(t)}{4},
$$ 
and the following conditions hold
\begin{align}
\label{eqzx0}
\frac{M_1}{\|y\|}&\le \lambda,\\
\label{eqzx1}
c_0(M_1 + a(t))&\le \frac{\lambda}{2a^2(t)}\bigg{[}a^2(t)-\frac{2|\dot{a}(t)|}{a(t)}\bigg{]},\\
\label{eqzx2}
c_1\frac{|\dot{a}(t)|}{a(t)}&\le \frac{a^2(t)}{2\lambda}\bigg{[}a^2(t)-\frac{2|\dot{a}(t)|}{a(t)}\bigg{]},\\
\label{eqzx3}
\frac{\lambda}{a^2(0)}g(0)&<1.
\end{align}
\end{lem}

\begin{proof}
Take 
\begin{equation}
\label{24eq3}
a(t) = \frac{d}{(c+t)^b},\quad 0<b\le\frac{1}{4},\quad 4b\le 
c^{1-2b}d^2,\quad c\ge1.
\end{equation} 
Note that $|\dot{a}|=-\dot{a}$.
We have
$$
\frac{|\dot{a}|}{a^3}=\frac{b}{d^2 (c+t)^{1-2b}}\le 
\frac{b}{d^2 c^{1-2b}}\le\frac{1}{4}.
$$
Hence,
\begin{equation}
\label{24eq2}
\frac{a^2(t)}{2}\le a^2(t)-\frac{2|\dot{a}(t)|}{a(t)} .
\end{equation}
Thus, inequality \eqref{eqzx1} is satisfied if
\begin{equation}
\label{24eq1}
c_0(M_1+a(0))\le\frac{\lambda}{4}.
\end{equation}
Take
\begin{equation}
\label{3eq19}
\lambda\ge \max\bigg{(}8c_0M_1,\frac{M_1}{\|y\|}\bigg{)}.
\end{equation} 
Then \eqref{eqzx0} is satisfied and
\begin{equation}
\label{23eq2}
c_0M_1\le\frac{\lambda}{8}.
\end{equation}
For any given $g(0)$, choose $a(0)$ sufficiently large so that
$$
\frac{\lambda}{ a^2(0)}g(0)<1.
$$
Then inequality \eqref{eqzx3} is satisfied. 

Choose $\kappa\ge1$ such that
\begin{equation}
\label{23eq3}
\kappa > \max\bigg{(}\sqrt{\frac{4\lambda 
c_1b}{d^4}},\frac{8c_0a(0)}{\lambda}, 1\bigg{)}.
\end{equation}
Define 
\begin{equation}
\label{23eq4}
\nu(t):=\kappa a(t),\qquad \lambda_\kappa:=\kappa^2 \lambda.
\end{equation}
Using inequalities \eqref{23eq2}, \eqref{23eq3} and \eqref{23eq4}, one gets
$$
c_0(M_1+\nu(0))\le \frac{\lambda}{8} + c_0\nu(0)\le \frac{\lambda_\kappa}{8}+
\frac{\lambda_\kappa}{8}=\frac{\lambda_\kappa}{4}.
$$
Thus, \eqref{24eq1} holds for $a(t)=\nu(t),\, \lambda=\lambda_\kappa$.
Consequently, \eqref{eqzx1} holds for $a(t)=\nu(t),\, \lambda=\lambda_\kappa$ since
\eqref{24eq2} holds as well under this transformation, i.e.,
\begin{equation}
\label{24eq4}
\frac{\nu^2(t)}{2}\leq \nu^2(t)-\frac {2|\dot{\nu}(t)|}{\nu(t)}.
\end{equation}

Using the inequalities \eqref{23eq3} and $c\ge 1$ and the definition 
\eqref{23eq4}, one obtains
$$
4\lambda_\kappa c_1\frac{|\dot{\nu }(t)|}{\nu ^5(t)}=4\lambda c_1\frac{b}{\kappa^2 d^4(c+t)^{1-4b}}
\le 4\lambda c_1\frac{b}{\kappa^2 d^4}
\le 1.
$$
This implies
$$
c_1\frac{|\dot{\nu }|}{\nu (t)}\le \frac{\nu^4 (t)}{4\lambda_\kappa}
\le\frac{\nu^2 (t)}{2\lambda_\kappa}\bigg{[}\nu ^2-\frac{2|\dot{\nu }|}{\nu }
\bigg{]}.
$$
Thus, one can replace the function $a(t)$ by $\nu(t)=\kappa a(t)$ and 
$\lambda$ by $\lambda_\kappa=\kappa^2 \lambda$ 
in the inequalities \eqref{eqzx0}--\eqref{eqzx3}.
\end{proof}

\begin{lem}
\label{lemma12}
Suppose $M_1, c_0, c_1$ and $\tilde{\alpha}$ are positive constants and $0\not=y\in H$.
Then there exist $\lambda>0$ and a sequence $0<(a_n)_{n=0}^\infty\searrow 0$ such that the following conditions hold
\begin{align}
\label{nshyeq22}
\frac{a_n}{a_{n+1}} &\le 2,\\
\label{nshyeq23}
\|f_\delta -F(0)\| &\le \frac{a_0^3}{\lambda},\\
\label{nshyeq24}
\frac{M_1}{\lambda}  &\le \|y\|,\\
\label{nshyeq25}
\frac{c_0(M_1+a_0)}{\lambda} &\le \frac{1}{2},\\
\label{nshyeq26}
\frac{a_n^2}{\lambda}-\frac{\tilde{\alpha} a_n^4}{2\lambda} + \frac{a_n-a_{n+1}}{a_{n+1}}c_1 &\le \frac{a_{n+1}^2}{\lambda}.
\end{align}
\end{lem}

\begin{proof}
Let us show that if $a_0>0$ is sufficiently large, then the following 
sequence
\begin{equation}
\label{nshpeq19}
a_n = \frac{a_0}{(1+n)^b},\qquad b = \frac{1}{4},
\end{equation}
satisfies conditions \eqref{nshyeq23}--\eqref{nshyeq26} 
%
if
\begin{equation}
\label{nshpeq20}
\lambda \ge \max\bigg{(}\frac{M_1}{\|y\|},4c_0M_1\bigg{)}.
\end{equation}
Condition \eqref{nshyeq22} is satisfied by the sequence \eqref{nshpeq19}. 
Inequality \eqref{nshyeq24} is satisfied since \eqref{nshpeq20} holds. 
Choose $a(0)$ so that
\begin{equation}
\label{nshqeq219}
a_0\ge \sqrt[3]{\|f_\delta -F(0)\| \lambda},
\end{equation}
then \eqref{nshyeq23} is satisfied. 


Assume that $(a_n)_{n=0}^\infty$ and $\lambda$ satisfy \eqref{nshyeq22}, \eqref{nshyeq23} and \eqref{nshyeq24}. 
Choose $\kappa\ge 1$ such that
\begin{equation}
\label{nshqeq21}
\kappa \geq\max \bigg{(}\frac{4c_0a_0}{\lambda},
\sqrt{\frac{4}{\tilde{\alpha} a_0^22\sqrt{2}}}
,\sqrt{\frac{\lambda c_1}{\tilde{\alpha}a_0^4}}
\bigg{)}.
\end{equation}
It follows from \eqref{nshqeq21} that
\begin{equation}
\label{26eq1}
\frac{4}{\kappa^2 a_0^22\sqrt{2}} \le \tilde{\alpha},\qquad \frac{\lambda c_1}{\kappa^2a_0^4}\le \tilde{\alpha}.
\end{equation}
Define 
\begin{equation}
\label{23eq5}
(b_n)_{n=0}^\infty:=(\kappa a_n)_{n=0}^\infty,\qquad \lambda_\kappa:=\kappa^2\lambda.
\end{equation} 
Using inequalities \eqref{nshpeq20}, \eqref{nshqeq21} and 
the definitions \eqref{23eq5}, one gets
$$
\frac{c_0(M_1+b_0)}{\lambda_\kappa } \le \frac{1}{4}+ 
\frac{c_0a_0}{\kappa \lambda}\le \frac{1}{4}+\frac{1}{4}=\frac{1}{2}.
$$
Thus, inequality \eqref{nshyeq25} holds for $a_0$ replaced by $b_0=\kappa 
a_0$ 
and $\lambda$ replaced by 
$\lambda_\kappa=\kappa^2\lambda$, where $\kappa$ satisfies \eqref{nshqeq21}.

For all $n\ge 0$ one has
\begin{equation}
\label{21eq1}
\frac{a_n^2 - a_{n+1}^2}{a_n^4}=\frac{a_n^4 - a_{n+1}^4}{a_n^4(a_n^2 + a_{n+1}^2)}\le \frac{a_n^4-a_{n+1}^4}{2a_{n+1}^2a_n^4}
= \frac{\frac{a_0^4}{n+1}-\frac{a_0^4}{n+2}}{2\frac{a_0^2}{\sqrt{n+2}}\frac{a_0^4}{n+1}}
=\frac{1}{a_0^22\sqrt{n+2}}\le \frac{1}{a_0^22\sqrt{2}}.
\end{equation}
Since $a_n$ is decreasing, one has
\begin{equation}
\label{21eq2}
\begin{split}
\frac{a_n - a_{n+1}}{a_n^4a_{n+1}}&= 
\frac{a_n^4 - a_{n+1}^4}{a_n^4a_{n+1}(a_n+a_{n+1})(a_n^2 + a_{n+1}^2)}\\
&\le \frac{a_n^4 - a_{n+1}^4}{4a_n^4a_{n+1}^4}
= \frac{\frac{a_0^4}{n+1}-\frac{a_0^4}{n+2}}{4\frac{a_0^4}{n+2}\frac{a_0^4}{n+1}}
 \le \frac{1}{4a_0^4},\qquad \forall n\ge0.
\end{split}
\end{equation}
Using inequalities 
\eqref{21eq1} and \eqref{26eq1}, one gets
\begin{equation}
\label{26eq2}
\frac{4(a_n^2-a_{n+1}^2)}{\kappa^2 a_n^4} \le \frac{4}{\kappa^2 a_0^22\sqrt{2}}\le \tilde{\alpha}.
\end{equation}
Similarly, using inequalities \eqref{21eq2} and \eqref{26eq1}, one gets
\begin{equation}
\label{26eq3}
\frac{4\lambda(a_n-a_{n+1})c_1}{\kappa^2a_n^4 a_{n+1}}\le \frac{\lambda c_1}{\kappa^2a_0^4}\le\tilde{\alpha}.
\end{equation}
Inequalities 
\eqref{26eq2} and \eqref{26eq3} imply
\begin{align*}
\frac{b_n^2-b_{n+1}^2}{\lambda_\kappa}+ \frac{b_n-b_{n+1}}{b_{n+1}}c_1 &= 
\frac{a_n^2-a_{n+1}^2}{\lambda}+ \frac{a_n-a_{n+1}}{a_{n+1}}c_1\\
&= \frac{\kappa^2a_n^4}{4\lambda} \frac{4(a_n^2-a_{n+1}^2)}{\kappa^2 a_n^4}+ 
\frac{\kappa^2a_n^4}{4\lambda}\frac{4\lambda(a_n-a_{n+1})c_1}{\kappa^2a_n^4 a_{n+1}}\\
&\le \frac{\kappa^2a_n^4}{4\lambda}\tilde{\alpha}+\frac{\kappa^2 a_n^4}{4\lambda} \tilde{\alpha}=
\frac{\kappa^2a_n^4 \tilde{\alpha}}{2\lambda}=\frac{\tilde{\alpha}b_n^4}{2\lambda_\kappa}.
\end{align*}
Thus, inequality \eqref{nshyeq26} holds for $a_n$ replaced by $b_n=\kappa a_n$ 
and $\lambda$ replaced by 
$\lambda_\kappa=\kappa^2\lambda$, where $\kappa$ satisfies \eqref{nshqeq21}. 
Inequalities \eqref{nshyeq22}--\eqref{nshyeq24} hold as well under this 
transformation. 
Thus, the choices $a_n=b_n$ and 
$\lambda:= \kappa\max\bigg{(}\frac{M_1}{\|y\|},4c_0M_1\bigg{)}$, 
where $\kappa$ satisfies \eqref{nshqeq21}, 
satisfy all the conditions of 
Lemma~\ref{lemma12}.
\end{proof}

\begin{rem}{\rm The constant $c_0$ and $c_1$ used in 
Lemma~\ref{lemma11} and \ref{lemma12} will be used in
Theorems \ref{mainthm} and \ref{mainthm2}. These constants are defined in 
equation \eqref{1eq10}.
The constant $\tilde{\alpha}$, used in Lemma~\ref{lemma12}, is the one from 
Theorem~\ref{mainthm2}.
This constant is defined in \eqref{25eq1}.
}
\end{rem}

\begin{rem}{\rm
\label{xrem}
Using similar arguments one can show that the 
sequence $a_n=\frac{d}{(c+n)^b}$, where $c\ge1$, $0<b\le \frac{1}{4},$ 
satisfy all conditions of Lemma~\ref{lemma4} 
provided that 
$d$ is sufficiently large and $\lambda$ is chosen so that inequality 
\eqref{nshpeq20} holds.
}
\end{rem}

\begin{rem}
\label{rem8}
{\rm
In the proof of Lemma~\ref{lemma12} and ~\ref{lemma11} the numbers $a_0$ and $\lambda$ can be chosen so 
that $\frac{a_0^2}{\lambda}$
is uniformly bounded as $\delta \to 0$ regardless of the rate
of growth of the constant $M_1=M_1(R)$ from formula \eqref{ceq3} when $R\to\infty$, 
i.e., regardless of the strength of the nonlinearity $F(u)$.
 
To satisfy \eqref{nshpeq20} one can choose $\lambda = M_1\big{(}\frac{1}{\|y\|}+4c_0\big{)}$.
To satisfy \eqref{nshqeq219} one can choose
$$
a_0 = \sqrt[3]{\lambda (\|f-F(0)\|+\|f\|)}\ge 
\sqrt[3]{\lambda \|f_\delta - F(0)\|},
$$
where we have assumed without loss of generality that $0<\|f_\delta - f\| < \|f\|$.
With this choice of $a_0$ and $\lambda$, the ratio $\frac {a^2_0}{\lambda}$
is bounded uniformly with respect to $\delta\in (0,1)$ and does not 
depend on $R$.
The dependence of $a_0$ on $\delta$ is seen from \eqref{nshqeq219}
since $f_\delta$ depends on $\delta$. 
In practice one  has $\|f_\delta-f\|<\|f\|$. Consequently, 
$$
\sqrt[3]{\|f_\delta -F(0)\| \lambda} 
\le \sqrt[3]{(\|f -F(0)\|+\|f\|) \lambda}.
$$
Thus, we can practically choose $a(0)$ independent of $\delta$ from the following inequality
$$
a_0 \ge \sqrt[3]{\lambda (\|f - F(0)\|+\|f\|)}.
$$
Indeed, with the above choice one has $\frac {a^2_0}{\lambda}\leq 
c(1+\sqrt[3]{\lambda^{-1}})\leq c$,
where $c>0$ is a constant independent of $\delta$, and one can 
assume that $\lambda\geq 1$ without loss of generality.

This Remark is used in the proof of the main result in Section \ref{mainsec}.
Specifically, it is used to prove that an iterative 
process \eqref{nsh3eq12} generates a sequence which stays in the ball
$B(u_0,R)$ for all $n\leq n_0 +1$, where the number $n_0$
is defined by formula \eqref{nsh4eq18} (see below), and $R>0$ is sufficiently large.
An upper bound on $R$ is given in the proof of Theorem \ref{mainthm2}, below formula 
\eqref{nshceq49}.
}
\end{rem}

\begin{rem}
\label{xrem2}
{\rm
One can choose $u_0\in H$ such that 
\begin{equation}
\label{nshteq20}
g_0:=\|u_0-V_0\|\le \frac{\|F(0)-f_\delta\|}{a_0}.
\end{equation}
Indeed, if, for example, $u_0=0$, then by Remark~\ref{remmoi} one gets
$$
g_0=\|V_0\|=\frac{a_0\|V_0\|}{a_0} \le \frac{\|F(0)-f_\delta\|}{a_0}.
$$
If \eqref{nshyeq23} and \eqref{nshteq20} hold then
$g_0 \le \frac{a_0^2}{\lambda}.$
}
\end{rem}

\section{Main results}
\label{mainsec}

\subsection{Dynamical systems gradient method}

Assume: 
\begin{equation}
\label{3eq11}
0<a(t)\searrow 0,\quad \lim_{t\to\infty}\frac{\dot{a}(t)}{a(t)}=0,
\quad \frac{|\dot{a}(t)|}{a^3(t)}\le \frac{1}{4}.
\end{equation}
Denote
$$
A:=F'(u_\delta(t)),\quad A_a:=A + aI,\quad a=a(t),
$$
where $I$ is the identity operator, 
and $u_\delta(t)$ solves the following Cauchy problem:
\begin{equation}
\label{3eq12}
\dot{u}_\delta = -A_{a(t)}^*[F(u_\delta)+a(t)u_\delta-f_\delta],\quad u_\delta(0)=u_0.
\end{equation}

\begin{thm}
\label{mainthm}
Assume that $F:H\to H$ is a monotone operator, twice Fr\'{e}chet
differentiable, $\sup_{u\in B(u_0,R)}\|F^{(j)}(u)\|\le M_j(R),\, 0\le j\le 2$,
$B(u_0,R):=\{u:\|u-u_0\|\le R\}$, $u_0$ is an element of $H$, satisfying inequality 
\eqref{deq47} (see below).
 Let $a(t)$ 
satisfy conditions of Lemma~\ref{lemma11}. For example, one 
can choose $a(t)=\frac{d}{(c+t)^b}$, 
where $b\in(0,\frac{1}{4}]$,\, 
$c\geq 1$, and $d>0$ are constants, and $d$ is sufficiently large.
Assume that equation $F(u)=f$ has a solution in $B(u_0,R)$, possibly nonunique,
and $y$ is the minimal-norm solution to this equation.
Let $f$ be unknown but $f_\delta$ be given, $\|f_\delta-f\|\le \delta$.
Then the solution $u_\delta(t)$ to problem \eqref{3eq12}
exists on an interval $[0,T_\delta]$,\, $\lim_{\delta\to0}T_\delta=\infty$, 
and 
there exists $t_\delta$, $t_\delta\in (0,T_\delta)$, not necessarily 
unique, such that 
\begin{equation}
\label{2eq3}
\|F(u_\delta(t_\delta))-f_\delta\|=C_1\delta^\zeta,
\quad \lim_{\delta\to 0}t_\delta=\infty,
\end{equation}
where $C_1>1$ and $0<\zeta\le 1$ are constants. If $\zeta\in (0,1)$ and 
$t_\delta$ satisfies \eqref{2eq3}, then
\begin{equation}
\label{2eq4}
\lim_{\delta\to 0} \|u_\delta(t_\delta) - y\|=0.
\end{equation}
\end{thm}

\begin{rem}{\rm
One can easily choose $u_0$ satisfying  inequality \eqref{deq47}.
Note that inequality \eqref{deq47} is a sufficient condition for 
\eqref{1eq20} to hold. In our proof inequality \eqref{1eq20}
is used at $t=t_\delta$.
The stopping time $t_\delta$ is often sufficiently large for 
the quantity $e^{-\varphi(t_\delta)}h_0$ to be  small. In this case  
inequality \eqref{1eq20} with $t=t_\delta$ is satisfied for a wide range of 
$u_0$. The parameter $\zeta$ is not fixed in \eqref{2eq3}. While we could 
fix it, for example, by setting $\zeta=0.9$, it is an interesting {\it 
open problem} to propose an optimal in some sense criterion for 
choosing $\zeta$.
}
\end{rem}

\begin{proof}[Proof of Theorem~\ref{mainthm}]
 Denote 
\begin{equation}
\label{beq18}
C:=\frac{C_1+1}{2}.
\end{equation}
Let 
$$
w:=u_\delta-V_\delta,\quad g(t):=\|w\|.
$$ 
One has
\begin{equation}
\label{1eq8}
\dot{w}=-\dot{V}_\delta-A_{a(t)}^*\big{[}F(u_\delta)-F(V_\delta)+a(t)w\big{]}.
\end{equation}
We use Taylor's formula and get:
\begin{equation}
\label{1eq9}
F(u_\delta)-F(V_\delta)+aw=A_a w+ K, \quad \|K\| \le\frac{M_2}{2}\|w\|^2,
\end{equation}
where $K:=F(u_\delta)-F(V_\delta)-Aw$, and $M_2$ is the constant from the estimate \eqref{ceq3}.
Multiplying \eqref{1eq8} by $w$ and using \eqref{1eq9} one gets
\begin{equation}
\label{3eq17}
g\dot{g}\le -a^2g^2+\frac{M_2(M_1+a)}{2}g^3+\|\dot{V}_\delta\|g,
\end{equation}
where the estimates: $\langle A_a^* A_a w,w\rangle \ge a^2 
g^2$ and
$\|A_a\|\le M_1+a$ were used. Note that the inequality $\langle A_a^* A_a 
w,w\rangle \ge a^2 g^2$ is true if $A\ge 0$.
Since $F$ is monotone and differentiable (see \eqref{ceq3}), 
one has $A:=F'(u_\delta)\ge 0$.

Let $t_0>0$ be such that 
\begin{equation}
\label{4eq18}
\frac{\delta}{a(t_0)}= \frac{1}{C-1}\|y\|,\qquad C>1.
\end{equation} 
This $t_0$ exists and is unique since $a(t)>0$ monotonically decays to 0 as $t\to\infty$.
By Lemma~\ref{lemma4},
there exists $t_1$ such that 
\begin{equation}
\label{3eq18}
\|F(V_\delta(t_1))-f_\delta\|=C\delta,\quad F(V_\delta(t_1))+a(t_1)V_\delta(t_1)-f_\delta=0.
\end{equation}
We claim that $t_1\in[0,t_0]$. 

Indeed, from \eqref{2eq2} and \eqref{2eq1} one gets
$$
C\delta=a(t_1)\|V_\delta(t_1)\|\le a(t_1)\bigg{(}\|y\|+ \frac{\delta}{a(t_1)}\bigg{)}=a(t_1)\|y\|+\delta,\quad C>1,
$$
so
$$
\delta\le \frac{a(t_1)\|y\|}{C-1}.
$$
Thus,
$$
\frac{\delta}{a(t_1)}\le \frac{\|y\|}{C-1}=\frac{\delta}{a(t_0)}.
$$
Since $a(t)\searrow 0$, the above inequality implies  $t_1\le t_0$. 
Differentiating both sides of \eqref{2eq2} with respect to $t$, one obtains
$$
A_{a(t)}\dot{V_\delta} = -\dot{a}V_\delta.
$$
This implies
\begin{equation}
\label{beq24}
\|\dot{V_\delta}\|\le |\dot{a}|\|A_{a(t)}^{-1} V_\delta\|\le \frac{|\dot{a}|}{a}\|V_\delta\|\le 
\frac{|\dot{a}|}{a}\bigg{(}\|y\|+\frac{\delta}{a}\bigg{)}\le \frac{|\dot{a}|}{a}\|y\|\bigg{(}1+\frac{1}{C-1}\bigg{)},\quad \forall t\le t_0.
\end{equation}
Since $g\ge 0$, inequalities \eqref{3eq17} and \eqref{beq24} imply
\begin{equation}
\label{1eq10}
\dot{g}\le -a^2(t)g(t)+c_0(M_1+a(t))g^2+\frac{|\dot{a}(t)|}{a(t)}c_1,
\quad c_0=\frac{M_2}{2},\, c_1=\|y\|\bigg{(}1+\frac{1}{C-1}\bigg{)}.
\end{equation}

Inequality \eqref{1eq10} is of the type \eqref{1eq7} with
$$
\gamma(t)=a^2(t),\quad \alpha(t)=c_0(M_1+a(t)),\quad \beta(t)=c_1\frac{|\dot{a}(t)|}{a(t)}.
$$
Let us check assumptions \eqref{1eq4}--\eqref{1eq6}. Take
$$
\mu(t)=\frac{\lambda}{a^2(t)},\quad \lambda =\text{const}.
$$
By Lemma~\ref{lemma11} there exist $\lambda$ and $a(t)$ such that conditions \eqref{1eq4}--\eqref{1eq6}
hold.
Thus, Lemma \ref{lemramm} yields
\begin{equation}
\label{roeq51}
g(t)<\frac{a^2(t)}{\lambda},\quad \forall t\le t_0.
\end{equation}
Therefore,
\begin{equation}
\label{1eq11}
\begin{split}
\|F(u_\delta(t))-f_\delta\|\le& \|F(u_\delta(t))-F(V_\delta(t))\|+\|F(V_\delta(t))-f_\delta\|\\
\le& M_1g(t)+\|F(V_\delta(t))-f_\delta\|\\
\le& \frac{M_1a^2(t)}{\lambda} + \|F(V_\delta(t))-f_\delta\|,\qquad \forall t\le t_0.
\end{split}
\end{equation}
It follows from Lemma~\ref{lemma3} that $\|F(V_\delta(t))-f_\delta\|$ is decreasing. 
Since $t_1\le t_0$, one gets 
\begin{equation}
\label{3eq21}
\|F(V_\delta(t_0))-f_\delta\|\le \|F(V_\delta(t_1))-f_\delta\|= C\delta.
\end{equation}
This, inequality \eqref{1eq11}, the inequality $\frac{M_1}{\lambda}\le \|y\|$ (see \eqref{3eq19}), the relation \eqref{4eq18},
 and the definition $C_1=2C-1$ (see \eqref{beq18}) imply
\begin{equation}
\label{1eq12}
\begin{split}
\|F(u_\delta(t_0))-f_\delta\| 
\le& \frac{M_1a^2(t_0)}{\lambda} + C\delta\\
\le& \frac{M_1\delta (C-1)}{\lambda\|y\|} + C\delta\le (2C-1)\delta=C_1\delta.
\end{split}
\end{equation}
We have used the inequality 
$$a^2(t_0)\le a(t_0)=\frac{\delta (C-1)}{\|y\|}$$ 
which
is true if $\delta$ is sufficiently small, or, equivalently, if $t_0$
is sufficiently large.
Thus, if 
$$
\|F(u_\delta(0))-f_\delta\|\ge C_1\delta^\zeta,\quad 0<\zeta\le 1,
$$
then there exists $t_\delta \in (0,t_0)$ such that
\begin{equation}
\label{3eq23}
\|F(u_\delta(t_\delta))-f_\delta\|=C_1\delta^\zeta
\end{equation}
for any given $\zeta\in (0,1]$, and any fixed $C_1>1$.

{\it Let us prove \eqref{2eq4}. If this is done, then Theorem 17 is 
proved.} 

First, we prove that $\lim_{\delta \to 0}\frac {\delta}{a(t_\delta)}=0.$

From \eqref{1eq11}  with $t=t_\delta$, and from \eqref{2eq1}, one gets
\begin{align*}
C_1\delta^\zeta &\le M_1 \frac{a^2(t_\delta)}{\lambda} + a(t_\delta)\|V_\delta(t_\delta)\|\\
&\le M_1 \frac{a^2(t_\delta)}{\lambda} + \|y\|a(t_\delta)+\delta.
\end{align*}
Thus, for sufficiently small $\delta$, one gets
$$
\tilde{C}\delta^\zeta \le a(t_\delta) \bigg{(}\frac{M_1a(0)}{\lambda}+\|y\|\bigg{)},\quad \tilde{C}>0,
$$
where $\tilde{C}<C_1$ is a constant. 
Therefore, 
\begin{equation}
\label{1eq14}
\lim_{\delta\to 0} \frac{\delta}{a(t_\delta)}\le
\lim_{\delta\to 0} \frac{\delta^{1-\zeta}}{\tilde{C}}\bigg{(}\frac{M_1a(0)}{\lambda}+\|y\|\bigg{)}
=0,\quad 0<\zeta<1.
\end{equation}

{\it Secondly, we prove that}
\begin{equation}
\label{1eq15}
\lim_{\delta\to0}t_\delta = \infty.
\end{equation}
Using \eqref{3eq12}, one obtains:
$$
\frac{d}{dt}\big{(}F(u_\delta)+au_\delta - f_\delta\big{)}= A_a\dot{u}_\delta + \dot{a}u_\delta
= -A_a A_a^*\big{(}F(u_\delta)+au_\delta - f_\delta\big{)} + \dot{a}u_\delta.
$$
This and \eqref{2eq2} imply:
\begin{equation}
\label{beq32}
\frac{d}{dt}\big{[}F(u_\delta)-F(V_\delta)+a(u_\delta-V_\delta)\big{]}
= - A_a A_a^*\big{[}F(u_\delta)-F(V_\delta) + a(u_\delta - V_\delta)\big{]} + \dot{a}u_\delta.
\end{equation}
Denote 
$$
v:=F(u_\delta)-F(V_\delta)+a(u_\delta-V_\delta),\quad h=\|v\|.
$$ 
Multiplying \eqref{beq32} by $v$ and using monotonicity of $F$, one obtains
\begin{equation}
\label{2eq5}
\begin{split}
h\dot{h} &= -\langle A_a A_a^* v, v\rangle +\langle v,\dot{a}(u_\delta-V_\delta)\rangle + \dot{a}\langle v,V_\delta\rangle\\ 
&\le -h^2a^2 + h|\dot{a}|\|u_\delta-V_\delta\| + |\dot{a}|h\|V_\delta\|,\qquad h\ge 0.
\end{split}
\end{equation}
Again, we have used the inequality $A_aA_a^* \ge a^2$, which holds for 
$A\geq 0$, i.e., monotone operators $F$.  
Thus,
\begin{equation}
\label{beq34}
\dot{h}\le -ha^2 + |\dot{a}|\|u_\delta - V_\delta\| + |\dot{a}|\|V_\delta\|.
\end{equation}
Since $\langle F(u_\delta)-F(V_\delta),u_\delta-V_\delta\rangle \ge 0$, one obtains two inequalities
\begin{equation}
\label{beq35}
a\|u_\delta - V_\delta\|^2 \le \langle v, u_\delta-V_\delta \rangle \le 
\|u_\delta - V_\delta\|h,
\end{equation}
and
\begin{equation}
\label{beq36}
\|F(u_\delta)-F(V_\delta)\|^2\le \langle v, F(u_\delta)-F(V_\delta) \rangle
\le h\|F(u_\delta)-F(V_\delta)\|.
\end{equation}
Inequalities \eqref{beq35} and \eqref{beq36} imply: 
\begin{equation}
\label{1eq16}
a\|u_\delta-V_\delta\|\le h,\quad \|F(u_\delta)-F(V_\delta)\|\le h.
\end{equation} 
Inequalities \eqref{beq34} and \eqref{1eq16} imply
\begin{equation}
\label{beq38}
\dot{h} \le -h\bigg{(}a^2-\frac{|\dot{a}|}{a}\bigg{)} +|\dot{a}|\|V_\delta\|. 
\end{equation}
Since $a^2-\frac{|\dot{a}|}{a}\ge \frac{3a^2}{4}>\frac{a^2}{2}$ 
by the last inequality in \eqref{3eq11}, 
it follows from inequality \eqref{beq38} that
\begin{equation}
\label{beq39}
\dot{h} \le -\frac{a^2}{2}h + |\dot{a}|\|V_\delta\|.
\end{equation}
Inequality \eqref{beq39} implies:
\begin{equation}
\label{1eq17}
h(t)\le h(0)e^{-\int_0^t\frac{a^2(s)}{2}ds} + e^{-\int_0^t\frac{a^2(s)}{2}ds}
\int_0^t e^{\int_0^s\frac{a^2(\xi)}{2}d\xi}|\dot{a}(s)|\|V_\delta(s)\|ds.
\end{equation}
Denote 
$$\varphi(t):=\int_0^t\frac{a^2(s)}{2}ds.$$ 
From \eqref{1eq17} and 
\eqref{1eq16}, one gets
\begin{equation}
\|F(u_\delta(t))-F(V_\delta(t))\| \le 
h(0)e^{-\varphi(t)} + e^{-\varphi(t)}\int_0^t 
e^{\varphi(s)}|\dot{a}(s)|\|V_\delta(s)\|ds.
\end{equation}
Therefore,
\begin{equation}
\label{1eq18}
\begin{split}
\|F(u_\delta(t))-f_\delta\|&\ge \|F(V_\delta(t))-f_\delta\|-\|F(V_\delta(t))-F(u_\delta(t))\|\\
&\ge a(t)\|V_\delta(t)\| - h(0)e^{-\varphi(t)} - 
e^{-\varphi(t)}\int_0^t e^{\varphi(s)} |\dot{a}| \|V_\delta\|ds.
\end{split}
\end{equation}
From Lemma~\ref{lemma9} it follows that there exists an 
$a(t)$ such that 
\begin{equation}
\label{1eq19}
\frac{1}{2}a(t)\|V_\delta(t)\| \ge  e^{-\varphi(t)}\int_0^te^{\varphi(s)}|\dot{a}| \|V_\delta(s)\|ds.
\end{equation}
For example, one can choose 
\begin{equation}
\label{ddeq47}
a(t)=\frac{c_1}{(c+t)^b}, \quad b\in (0,\frac{1}{4}],\quad c_1^2 c^{1-2b}\ge 6b,
\end{equation}
where 
$c_1,c>0$.
Moreover, one can always choose $u_0$ such that 
\begin{equation}
\label{deq47}
h(0)=\|F(u_0) + a(0)u_0 -f_\delta\| \le \frac{1}{4} a(0)\|V_\delta(0)\|,
\end{equation}
because the equation 
$$
F(u_0) + a(0)u_0 -f_\delta=0
$$ 
is solvable.

If \eqref{deq47} holds, then
\begin{equation}
\label{23eq6}
h(0)e^{-\varphi(t)}\le \frac{1}{4}a(0)\|V_\delta(0)\|e^{-\varphi(t)},\qquad 
t\ge 0.
\end{equation}
If \eqref{ddeq47} holds, $c\ge1$  and $2b\le c_1^2$, then it follows that 
\begin{equation}
\label{26eq5}
e^{-\varphi(t)}a(0)\le a(t).
\end{equation}
Indeed, inequality $a(0)\le a(t)e^{\varphi(t)}$ is obviously true for 
$t=0$, and $\big(a(t)e^{\varphi(t)}\big)'_t\geq 0$, provided that $c\geq 
1$ and $2b\le c_1^2$. 

Inequalities \eqref{23eq6} and \eqref{23eq5} imply
\begin{equation}
\label{1eq20}
e^{-\varphi(t)}h(0) 
\le \frac{1}{4} a(t)\|V_\delta(0)\|
\le \frac{1}{4} a(t)\|V_\delta(t)\|,\quad t\ge 0.
\end{equation}
where we have used the inequality $\|V_\delta(t)\|\le \|V_\delta(t')\|$ for $t\le t'$, 
established in Lemma~\ref{lemma3}.
From \eqref{3eq23} and \eqref{1eq18}--\eqref{1eq20}, one gets
$$
C\delta^\zeta = \|F(u_\delta(t_\delta))-f_\delta\|\ge \frac{1}{4}a(t_\delta)\|V_\delta(t_\delta)\|.
$$
Thus,
$$
\lim_{\delta\to0}a(t_\delta)\|V_\delta(t_\delta)\|\le 
\lim_{\delta\to0}4C\delta^\zeta = 0.
$$
Since $\|V_\delta(t)\|$ is increasing, this implies 
$\lim_{\delta\to0}a(t_\delta)=0$. 
Since $0<a(t)\searrow 0$, it follows that  
\eqref{1eq15} holds. 

From the triangle inequality and inequalities \eqref{roeq51} and \eqref{rejected11} one obtains
\begin{equation}
\label{eqhic56}
\begin{split}
\|u_\delta(t_\delta) - y\| &\le \|u_{\delta}(t_\delta) - V_{\delta}\| + 
\|V(t_\delta) - V_\delta(t_\delta)\| + \|V(t_\delta) - y\|\\
&\le \frac{a^2(t_\delta)}{\lambda} + \frac{\delta}{a(t_\delta)} + \|V(t_\delta)-y\|.
\end{split}
\end{equation}
From \eqref{1eq14}, \eqref{1eq15}, inequality \eqref{eqhic56} and Lemma~\ref{rejectedlem}, one obtains
\eqref{2eq4}. Theorem~\ref{mainthm} is proved.
\end{proof}

\subsection{An iterative scheme}

Let $V_{n,\delta}$ solve the equation:
$$
F(V_{n,\delta}) + a_n V_{n,\delta} - f_\delta = 0.
$$
Denote $V_n:=V_{n,\delta}$. 

Consider the following iterative scheme:
\begin{equation}
\label{nsh3eq12}
\begin{split}
u_{n+1} &= u_n - \alpha_n A_n^*[F(u_n)+a_n u_n - f_\delta],\quad A_n:=F'(u_n)+ a_nI,\quad u_0=u_0,
\end{split}
\end{equation}
where $u_0$ is chosen so that inequality \eqref{nshteq20} holds, and $\{\alpha_n\}_{n=1}^\infty$
is a positive sequence such that
\begin{equation}
\label{25eq1}
0<\tilde{\alpha}\le \alpha_n \le \frac{2}{a_n^2 + (M_1+a_n)^2},
\qquad ||A_n||\leq M_1+a_n.
\end{equation}
It follows from this condition that
\begin{equation}
\label{25eq2}
\|1-\alpha_n A_{a_n}^*A_{a_n}\|= \sup_{a_n^2\leq \lambda \leq 
(M_1+a_n)^2}|1-\alpha_n \lambda| 
\le 1 - \alpha_n a_n^2.
\end{equation} 
Note that  $F'(u_n)\ge 0$ since $F$ is monotone. 

Let $a_n$ and $\lambda$ 
satisfy conditions \eqref{nshyeq22}--\eqref{nshyeq26}.
Assume that equation $F(u)=f$ has a solution in $B(u_0,R)$, possibly nonunique,
and $y$ is the minimal-norm solution to this equation. 
Let $f$ be unknown but $f_\delta$ be given, and $\|f_\delta-f\|\le \delta$.
We prove the following result:

\begin{thm}
\label{mainthm2}
Assume $a_n=\frac{d}{(c+n)^b}$ where $c\ge 1,\, 0<b\le \frac{1}{4}$, and $d$ is sufficiently large
so that conditions \eqref{nshyeq22}--\eqref{nshyeq26} hold. 
Let $u_n$ be defined by \eqref{nsh3eq12}. Assume that $u_0$ is chosen so that \eqref{nshteq20} holds. 
Then there exists a unique $n_\delta$ such that
\begin{equation}
\label{nsh2eq3}
\|F(u_{n_\delta})-f_\delta\|\le C_1\delta^\zeta,\quad
C_1\delta^\zeta < \|F(u_{n})-f_\delta\|,\quad \forall n< n_\delta,
\quad 
\end{equation}
where $C_1>1,\, 0<\zeta\le 1$.

Let $0<(\delta_m)_{m=1}^\infty$ be a sequence such that $\delta_m\to 0$. 
If the sequence $\{n_m:=n_{\delta_m}\}_{m=1}^\infty$ is bounded, and $\{n_{m_j}\}_{j=1}^\infty$ 
is a convergent subsequence, then
\begin{equation}
\label{nshfeq15}
\lim_{j\to\infty} u_{n_{m_j}} = \tilde{u},
\end{equation}
where $\tilde{u}$ is a solution to the equation $F(u)=f$.
If 
\begin{equation}
\label{nshfeq16}
\lim_{m\to \infty}n_m=\infty,
\end{equation}
where  $\zeta\in (0,1)$, then
\begin{equation}
\label{nshfeq17}
\lim_{m\to \infty} \|u_{n_m} - y\|=0.
\end{equation}
\end{thm}

\begin{proof}
Denote 
\begin{equation}
\label{nshbeq18}
C:=\frac{C_1+1}{2}.
\end{equation}
Let 
$$
z_n:=u_n-V_n,\quad g_n:=\|z_n\|.
$$ 
We use Taylor's formula and get:
\begin{equation}
\label{nsh1eq9}
F(u_n)-F(V_n)+a_nz_n=A_{n} z_n+ K_n, \quad \|K_n\| \le\frac{M_2}{2}\|z_n\|^2,
\end{equation}
where $K_n:=F(u_n)-F(V_n)-F'(u_n)z_n$ and $M_2$ is the constant 
from \eqref{ceq3}.
From \eqref{nsh3eq12} and \eqref{nsh1eq9} one obtains
\begin{equation}
\label{nsh1eq8}
z_{n+1} = z_n - \alpha_n A_n^*A_nz_n - \alpha_n A_n^*K(z_n) - (V_{n+1}-V_{n}).
\end{equation}
From \eqref{nsh1eq8}, \eqref{nsh1eq9}, \eqref{25eq2}, and the 
estimate $\|A_n\|\le M_1+a_n$, one gets
\begin{equation}
\label{nsh3eq17}
\begin{split}
g_{n+1}   &\le g_n\|1 - \alpha_n A_n^*A_n\|+ \frac{\alpha_n M_2(M_1+a_n)}{2}g_n^2 + \|V_{n+1}-V_n\|\\
          &\le g_n(1-\alpha_n a_n^2)+\frac{\alpha_n M_2(M_1+a_n)}{2}g_n^2 
+ \|V_{n+1}-V_n\|.
\end{split}
\end{equation}
Since $0<a_n\searrow 0$, for any fixed $\delta>0$ there exists $n_0$ such that
\begin{equation}
\label{nsh4eq18}
\frac{\delta}{a_{n_0+1}}> \frac{1}{C-1}\|y\|\ge \frac{\delta}{a_{n_0}},\qquad C>1.
\end{equation} 
By \eqref{nshyeq22}, one has $\frac{a_n}{a_{n+1}}\le 2,\, \forall\, n\ge 0$. This and \eqref{nsh4eq18} imply
\begin{equation}
\label{nsheeq16}
\frac{2}{C-1}\|y\|\ge \frac{2\delta}{a_{n_0}} >\frac{\delta}{a_{n_0+1}}>
 \frac{1}{C-1}\|y\|\ge \frac{\delta}{a_{n_0}},\qquad C>1.
\end{equation}
Thus,
\begin{equation}
\label{nshceq18}
\frac{2}{C-1}\|y\|> \frac{\delta}{a_{n}},\quad \forall n \le n_0 + 1.
\end{equation}
The number $n_0$, satisfying \eqref{nshceq18}, exists and is unique since $a_n>0$ monotonically decays to 0 as $n\to\infty$.
By Remark~\ref{rem2.9},
there exists  a number $n_1$ such that 
\begin{equation}
\label{nsh3eq18}
\|F(V_{n_1+1})-f_\delta\|\le C\delta < \|F(V_{n_1})-f_\delta\|, 
\end{equation}
where $V_n$ solves the equation $F(V_{n})+a_{n}V_{n}-f_\delta=0$. 

{\it We claim that $n_1\in[0,n_0]$.} 

Indeed, 
one has $\|F(V_{n_1})-f_\delta\|=a_{n_1}\|V_{n_1}\|$, and $\|V_{n_1}\|\le \|y\|+\frac{\delta}{a_{n_1}}$ 
(cf. \eqref{2eq1}), so
\begin{equation}
\label{nsheeq19}
C\delta < a_{n_1}\|V_{n_1}\|\le a_{n_1}\bigg{(}\|y\|+ \frac{\delta}{a_{n_1}}\bigg{)}=a_{n_1}\|y\|+\delta,\quad C>1.
\end{equation}
Therefore,
\begin{equation}
\label{nsheeq20}
\delta < \frac{a_{n_1}\|y\|}{C-1}.
\end{equation}
Thus, by \eqref{nsheeq16}, 
\begin{equation}
\label{nshyeq36}
\frac{\delta}{a_{n_1}} < \frac{\|y\|}{C-1} < \frac{\delta}{a_{n_0+1}}.
\end{equation}
Here the last inequality is a consequence of \eqref{nsheeq16}.
Since $a_n$ decreases monotonically, inequality \eqref{nshyeq36} implies $n_1\le n_0$. 
One has
\begin{equation}
\label{nsheeq21}
\begin{split}
a_{n+1} \|V_n-V_{n+1}\|^2 &= \langle (a_{n+1} - a_{n})  V_{n} - F(V_n) + F(V_{n+1}), V_n - V_{n+1} \rangle \\
&\le \langle (a_{n+1} - a_{n})  V_{n}, V_n - V_{n+1} \rangle \\
&\le (a_{n}-a_{n+1}) \|V_{n}\| \|V_n - V_{n+1}\|.
\end{split}
\end{equation}
By \eqref{2eq1}, $\|V_n\|\le \|y\|+\frac{\delta}{a_n}$, and, by \eqref{nshceq18}, 
$\frac{\delta}{a_n}\le \frac{2\|y\|}{C-1}$ for all $ n\le n_0+1$.
Therefore,  
\begin{equation}
\label{nshceq40}
\|V_n\|\le \|y\|\bigg{(}1+\frac{2}{C-1}\bigg{)},\qquad \forall n\le n_0+1,
\end{equation}
and, by \eqref{nsheeq21}, 
\begin{equation}
\label{nshbeq24}
\|V_n-V_{n+1}\| \le \frac{a_n-a_{n+1}}{a_{n+1}}\|V_{n}\|\le \frac{a_n-a_{n+1}}{a_{n+1}}\|y\|
\bigg{(}1+\frac{2}{C-1}\bigg{)},\quad \forall n\le n_0+1.
\end{equation}
Inequalities \eqref{nsh3eq17} and \eqref{nshbeq24} imply
\begin{equation}
\label{nsh1eq10}
g_{n+1}\le (1-\alpha_n a_n^2)g_n + \alpha_n c_0(M_1+a_n)g_n^2+\frac{a_n-a_{n+1}}{a_{n+1}}c_1,\qquad \forall \, n\le n_0+1,
\end{equation}
where the constants $c_0$ and $c_1$ are defined in \eqref{1eq10}.

By Lemma~\ref{lemma4} and Remark~\ref{xrem}, the sequence 
$(a_n)_{n=1}^\infty$, satisfies conditions \eqref{nshyeq22}--\eqref{nshyeq26}, 
provided that $a_0$ is sufficiently large and 
$\lambda>0$ is 
chosen so that \eqref{nshpeq20} holds.
Let us show by induction that 
\begin{equation}
\label{nshceq15}
g_n<\frac{a_n^2}{\lambda},\qquad 0\le n\le n_0+1.
\end{equation}
Inequality \eqref{nshceq15} holds for $n=0$ by Remark~\ref{xrem2}. Suppose \eqref{nshceq15} holds for some $n\ge 0$. 
From \eqref{nsh1eq10}, \eqref{nshceq15} and \eqref{nshyeq26}, one gets
\begin{equation}
\begin{split}
g_{n+1}&\le (1-\alpha_n a_n^2)\frac{a_n^2}{\lambda}+ 
\alpha_n c_0(M_1+a_n)\bigg{(}\frac{a_n^2}{\lambda}\bigg{)}^2 + \frac{a_n-a_{n+1}}{a_{n+1}}c_1\\
&=\frac{a_n^4}{\lambda}\bigg{(}\frac{\alpha_n c_0(M_1+a_n)}{\lambda}-\alpha_n\bigg{)}+ 
\frac{a_n^2}{\lambda} + \frac{a_n-a_{n+1}}{a_{n+1}}c_1\\
&\le -\frac{\alpha_n a_n^4}{2\lambda}+ 
\frac{a_n^2}{\lambda} + \frac{a_n-a_{n+1}}{a_{n+1}}c_1\\
&\le \frac{a_{n+1}^2}{\lambda}.
\end{split}
\end{equation}
Thus, by induction, inequality \eqref{nshceq15} holds for all $n$ in the region $0\le n\le n_0+1$.

From \eqref{2eq1} one has $\|V_n\| \le \|y\|+\frac{\delta}{a_n}$. 
This and the triangle inequality imply 
\begin{equation}
\label{nshceq49}
\|u_0-u_n\| \le \|u_0\|+ \|z_n\|+ \|V_n\|\le \|u_0\|+\|z_n\|+ \|y\|+\frac{\delta}{a_n}.
\end{equation}
Inequalities \eqref{nshceq40}, \eqref{nshceq15},
and \eqref{nshceq49} guarantee that the sequence $u_n$, generated by the 
iterative process \eqref{nsh3eq12}, remains
in the ball $B(u_0,R)$ for all $n\le n_0+1$, where 
$R\le \frac{a_0}{\lambda}+\|u_0\|+\|y\|+ \frac{\delta}{a_n}$.
This inequality and the estimate \eqref{nshceq18} imply that the sequence 
$u_n$, $n\le 
n_0+1,$ stays in the ball $B(u_0,R)$,
where 
\begin{equation}
\label{26eq6}
R\le \frac{a_0}{\lambda}+ \|u_0\|+\|y\|+ \|y\|\frac{C+1}{C-1}.
\end{equation}
By Remark~\ref{rem8}, one can choose $a_0$ and $\lambda$ so that 
$\frac{a_0}{\lambda}$
is uniformly bounded as $\delta \to 0$ even if $M_1(R)\to\infty$ as 
$R\to\infty$ at an arbitrary fast rate.
Thus, the sequence $u_n$ stays in the ball $B(u_0,R)$ for $n\leq n_0+1$
when $\delta\to 0$. An upper bound on
$R$ is given above. It does not depend on $\delta$ as 
$\delta\to 0$.

One has:
\begin{equation}
\label{nsh1eq11}
\begin{split}
\|F(u_n)-f_\delta\|\le& \|F(u_n)-F(V_n)\|+\|F(V_n)-f_\delta\|\\
\le& M_1g_n+\|F(V_n)-f_\delta\|\\
\le& \frac{M_1a_n^2}{\lambda} + \|F(V_n)-f_\delta\|,\qquad \forall n\le n_0+1,
\end{split}
\end{equation}
where \eqref{nshceq15} was used and $M_1$ is the constant from \eqref{ceq3}. 
Since $\|F(V_n)-f_\delta\|$ is decreasing, by Lemma~\ref{lemma3}, and
$n_1\le n_0$, one gets 
\begin{equation}
\label{nsh3eq21}
\|F(V_{n_0+1})-f_\delta\|\le \|F(V_{n_1+1})-f_\delta\| \le C\delta.
\end{equation}
From \eqref{nshyeq24}, \eqref{nsh1eq11}, \eqref{nsh3eq21}, 
the relation \eqref{nsh4eq18}, 
and the definition $C_1=2C-1$ (see \eqref{nshbeq18}), one concludes that
\begin{equation}
\label{nsh1eq12}
\begin{split}
\|F(u_{n_0+1})-f_\delta\| 
\le& \frac{M_1a_{n_0+1}^2}{\lambda} + C\delta \\
\le& \frac{M_1\delta (C-1)}{\lambda\|y\|} + C\delta\le (2C-1)\delta=C_1\delta.
\end{split}
\end{equation}
{\it Thus, if 
$$
\|F(u_0)-f_\delta\|> C_1\delta^\zeta,\quad 0<\zeta\le 1,
$$
then one concludes from \eqref{nsh1eq12} that there exists 
$n_\delta$, $0<n_\delta \le 
n_0+1,$ such that
\begin{equation}
\label{nsh3eq23}
\|F(u_{n_\delta})-f_\delta\| \le C_1\delta^\zeta < \|F(u_{n})-f_\delta\|,\quad 0\le n< n_\delta,
\end{equation}
for any given $\zeta\in (0,1]$, and any fixed $C_1>1$.}

{\it Let us prove \eqref{nshfeq15}.}

 If $n>0$ is fixed, then 
$u_{\delta,n}$ is a
continuous function of $f_\delta$. Denote
\begin{equation}
\label{nshceq46}
\tilde{u}:=\tilde{u}_N=\lim_{\delta\to 0}u_{\delta,n_{m_j}},
\end{equation}
where 
$$
\lim_{j\to\infty}n_{m_j} = N.
$$
From \eqref{nshceq46} and the continuity of $F$, one obtains:
$$
\|F(\tilde{u})-f_\delta\| = \lim_{j\to\infty}\|F(u_{n_{m_j}})-f_\delta\|\le \lim_{\delta\to 0}C_1\delta^\zeta = 0.
$$
Thus, $\tilde{u}$ is a solution to the equation $F(u)=f$, and \eqref{nshfeq15} is proved.

{\it Let us prove \eqref{nshfeq17} assuming that \eqref{nshfeq16} holds.} 

From \eqref{nsh2eq3} and \eqref{nsh1eq11}  with $n=n_\delta-1$, 
and from \eqref{nsh3eq23}, one gets
\begin{align*}
C_1\delta^\zeta &\le M_1 \frac{a_{n_\delta-1}^2}{\lambda} + 
a_{n_\delta-1}\|V_{n_\delta-1}\|
\le M_1 \frac{a_{n_\delta-1}^2}{\lambda} + \|y\|a_{n_\delta-1}+\delta.
\end{align*}
If $\delta>0$ is sufficiently small, then the above equation implies
$$
\tilde{C}\delta^\zeta \le a_{n_\delta-1} \bigg{(}\frac{M_1a_0}{\lambda}+\|y\|\bigg{)},\quad \tilde{C}>0,
$$
where $\tilde{C}<C_1$ is a constant, and the inequality $ 
a^2_{n_\delta-1}\le a_{n_\delta-1}a_0$ was used. 
Therefore, by \eqref{nshyeq22}, 
\begin{equation}
\label{nsh1eq14}
\lim_{\delta\to 0} \frac{\delta}{2a_{n_\delta}}\le
\lim_{\delta\to 0} \frac{\delta}{a_{n_\delta-1}}\le
\lim_{\delta\to 0} \frac{\delta^{1-\zeta}}{\tilde{C}}\bigg{(}\frac{M_1a_0}{\lambda}+\|y\|\bigg{)}
=0,\quad 0<\zeta<1.
\end{equation}
In particular, for $\delta=\delta_m$, one gets
\begin{equation}
\label{1eq144}
\lim_{\delta_m\to 0} \frac{\delta_m}{a_{n_m}} = 0.
\end{equation}
From the triangle inequality and inequalities \eqref{rejected11} and \eqref{nshceq15} one obtains
\begin{equation}
\label{eqzx56}
\begin{split}
\|u_{n_m} - y\| &\le \|u_{n_m} - V_{n_m}\| + 
\|V_n - V_{{n_m},0}\| + \|V_{{n_m},0}-y\|\\
&\le \frac{a^2_{n_m}}{\lambda} + \frac{\delta_m}{a_{n_m}} + \|V_{{n_m},0}-y\|.
\end{split}
\end{equation}
From \eqref{nshfeq16}, \eqref{1eq144}, inequality \eqref{eqzx56} and Lemma~\ref{rejectedlem}, one obtains
\eqref{nshfeq17}. Theorem~\ref{mainthm2} is proved.
\end{proof}

\section{Numerical experiments}
\label{numsec}

Let us do a numerical experiment solving nonlinear equation \eqref{aeq1} with 
\begin{equation}
\label{1eq41}
F(u):= B(u)+\frac{u^3} {6}:=\int_0^1e^{-|x-y|}u(y)dy + \frac{u^3}{6},
 \quad f(x): = \frac {13}{6}-e^{-x}-\frac{e^x}{e}.
\end{equation}
Such equation is a model nonlinear equation in Wiener-type filtering 
theory, see \cite{R486}.

One can check that $u(x)\equiv 1$ solves  the equation $F(u)=f$.
The operator $B$ is compact in $H=L^2[0,1]$. The operator $u\longmapsto u^3$
is defined on a dense subset $D$ of of $L^2[0,1]$, for example, on $D:=C[0,1]$.
If $u,v\in D$, then 
$$
\langle u^3-v^3,u-v\rangle = \int_0^1(u^3-v^3)(u-v)dx \ge 0.
$$ 
Moreover, 
$$
e^{-|x|} = \frac{1}{\pi}\int_{-\infty}^\infty \frac{e^{i\lambda x}}{1+\lambda^2} d\lambda.
$$
Therefore, $\langle B(u-v),u-v\rangle\ge0$, so 
$$
\langle F(u-v),u-v\rangle\ge0,\qquad \forall u,v\in D.
$$

Note that $D$ does not contain subsets, open in $H=L^2[0,1]$, i.e., it does not contain interior points of $H$.
This is a reflection of the fact that the operator $G(u)=\frac{u^3}{6}$
is unbounded on any open subset of $H$.
For example, in any ball $\|u\|\le C$, $C=const>0$, where $\|u\|:=\|u\|_{L^2[0,1]}$, there is an element $u$ such that
$\|u^3\|=\infty$. As such an element one can take, for example, $u(x)=c_1 x^{-b}$, $\frac{1}{3}<b<\frac{1}{2}$. here $c_1>0$ is a constant chosen so that
$\|u\|\leq C$.
The operator $u\longmapsto F(u)=G(u)+B(u)$ is maximal monotone on $D_F:=\{u:u\in H,\, F(u)\in H\}$ (see \cite[p.102]{D}),
so that equation \eqref{2eq2} is uniquely solvable for any $f_\delta\in H$.

The Fr\'{e}chet derivative of $F$ is:
\begin{equation}
\label{eq44}
F'(u)h = \frac{u^2 h}{2} + \int_0^1 e^{-|x-y|}h(y) dy.
\end{equation}
If $u(x)$ vanishes on a set of positive Lebesgue's measure, then $F'(u)$
is obviously not boundedly invertible.
If $u\in C[0,1]$ vanishes even at one point $x_0$, then $F'(u)$ is not boundedly invertible in $H$.

Let us use the 
iterative process \eqref{nsh3eq12}:
\begin{equation}
\begin{split}
u_{n+1} &= u_n - \alpha_n \big{(}F'(u_n)^*+a_nI \big{)}(F(u_n)+a_nu_n - f_\delta),\\
u_0 &= 0.
\end{split}
\end{equation}
We stop iterations at $n:=n_\delta$ such that the following inequality holds
\begin{equation}
\label{eq53}
\|F(u_{n_\delta}) - f_\delta\| <C \delta^\zeta,\quad
\|F(u_{n}) - f_\delta\|\ge C\delta^\zeta,\quad n<n_\delta ,\quad C>1,\quad \zeta \in(0,1).
\end{equation}
Integrals of the form 
$\int_0^1 e^{-|x-y|}h(y)dy$ in \eqref{1eq41} and \eqref{eq44} are computed by 
using
the trapezoidal rule. The noisy function used in the test is
$$
f_\delta(x) = f(x) + \kappa f_{noise}(x),\quad \kappa>0.
$$
The noise level $\delta$ and the relative noise level are determined by 
$$
\delta = \kappa\| f_{noise}\|,\quad \delta_{rel}:=\frac{\delta}{\|f\|}.
$$
In the test, $\kappa$ is computed in such a way that the relative noise level
$\delta_{rel}$ equals to some desired value, i.e.,
$$
\kappa = \frac{\delta}{\| f_{noise}\|}=\frac{\delta_{rel}\|f\|}{\| f_{noise}\|}.
$$
We have used the relative noise level as an input parameter in the test.

The version of DSM, developed in this paper and denoted by DSMG, is compared with the version of DSM in \cite{546},
denoted by DSMN. Indeed, the DSMN is the following iterative scheme
\begin{equation}
u_{n+1} = u_n - A_n^{-1}\big{(}F'(u_n) + a_n u_n -f_\delta\big{)},\quad u_0 = u_0,\qquad n\ge 0,
\end{equation}
where $a_n = \frac{a_0}{1+n}$. This iterative scheme is used with a stopping time $n_\delta$
defined by \eqref{nsh2eq3}. The existence of this stopping time and the convergence of the method is
proved in \cite{546}.

As we have proved, the DSMG converges when $a_n=\frac{a_0}{(1+n)^b},\, b\in (0,\frac{1}{4}],$
and $a_0$ is sufficiently large. However, in practice, if we choose $a_0$ too large then the method will use 
too many iterations before reaching the stopping time $n_\delta$ in \eqref{eq53}. This means that the computation time
is large. Since
$$
\|F(V_\delta) - f_\delta\| = a(t)\|V_\delta\|,
$$ 
and $\|V_\delta(t_\delta) - u_\delta(t_\delta)\|=O(a(t_\delta))$, we have
$$
C\delta^\zeta =\|F(u_\delta(t_\delta)) - f_\delta\| \sim a(t_\delta).
$$
Thus, we choose
$$
a_0 = C_0\delta^\zeta,\qquad C_0>0.
$$
The parameter $a_0$ used in the DSMN is also chosen by this formula.

In all figures, the $x$-axis represents the variable $x$. In all figures, by {\it DSMG} we denote 
the numerical solutions obtained by the DSMG, by {\it DSMN} we denote
solutions by the DSMN and by {\it exact} we denote the exact solution.

In experiments, we found that the DSMG works well with $a_0 = C_0\delta^\zeta$, $C_0\in[0.2,1]$. 
Indeed, in the test the DSMG is implemented with $a_n := C_0\frac{\delta^{0.99}}{(n+1)^{0.25}}$, $C_0=0.5$
while the DSMN is implemented with $a_n := C_0\frac{\delta^{0.99}}{(n+1)}$, $C_0=1$.
 For $C_0>1$ the convergence rate of DSMG is much slower while the DSMN still works well if $C_0\in [1,4]$. 
 
Figure~\ref{figmono} plots the solutions using relative noise levels 
$\delta=0.01$ and $\delta = 0.001$. 
The exact solution used in these experiments is $u=1$. 
In the test the DSMG is implemented with $\alpha_n=1$, $C = 1.01$, $\zeta = 0.99$ 
and $\alpha_n=1,\,\forall n\ge0$. 
The number of iterations of the DSMG
for $\delta=0.01$ and $\delta = 0.001$ were 49 and 50
while the number of iteration for the DSMN are 9 and 9, respectively. 
The number of node points used in computing integrals in \eqref{1eq41} and \eqref{eq44}
was $N = 100$. The noise function $f_{noise}$ in this experiment is a vector with random entries 
normally distributed of mean 0 and variant 1. 
Figure~\ref{figmono} shows that the solutions by the DSMN and DSMG are nearly the same in this
figure.

\begin{figure}[!h!tb]
\centerline{%
\includegraphics[scale=0.95]{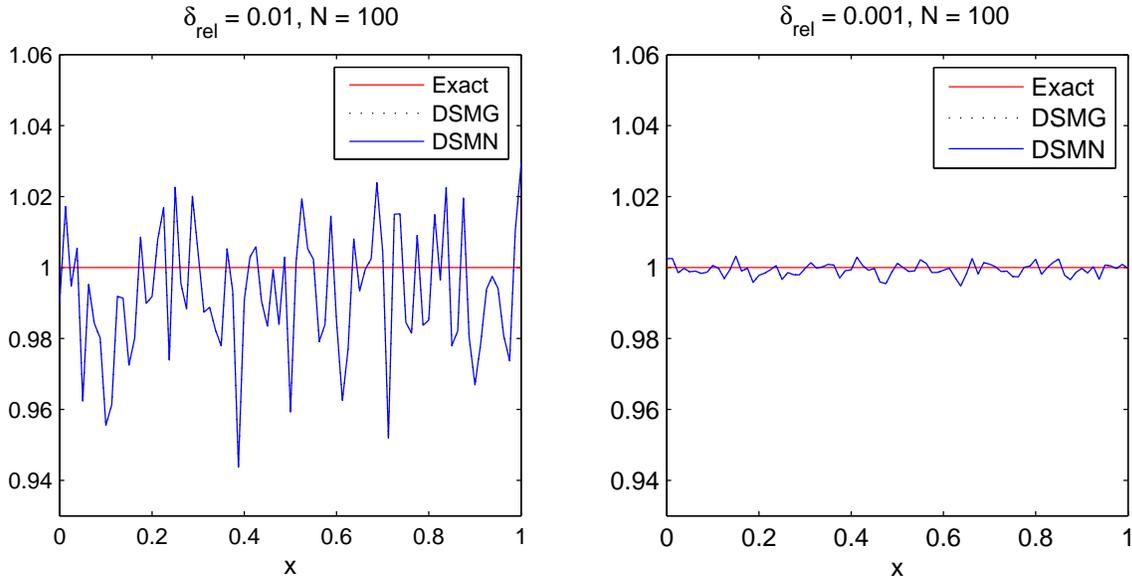}
}
\caption{Plots of solutions obtained by the DSMN and DSMG when $N = 100$, $u=1$, $x\in[0,1]$, $\delta_{rel}=0.01$ (left) 
and $N = 100$, $u=1$, $x\in[0,1]$, $\delta_{rel}=0.001$ (right).}
\label{figmono}
\end{figure}

Figure~\ref{figmono2} presents the numerical results when $N = 100$ with $\delta=0.01$ $u(x)=\sin(2\pi x)$, $x\in[0,1]$
(left) and with $\delta = 0.01$, $u(x)=\sin(\pi x)$, $x\in[0,1]$ (right).
In these cases, the DSMN took 11 and 7 iterations to give the numerical solutions while
the DSMG took 512 and 94 iterations for
$u(x)=\sin(2\pi x)$ and $u(x)=\sin(\pi x)$, respectively.
Figure~\ref{figmono2} show that the numerical results of the DSMG are better than those of the DSMN.

\begin{figure}[!h!tb]
\centerline{%
\includegraphics[scale=0.96]{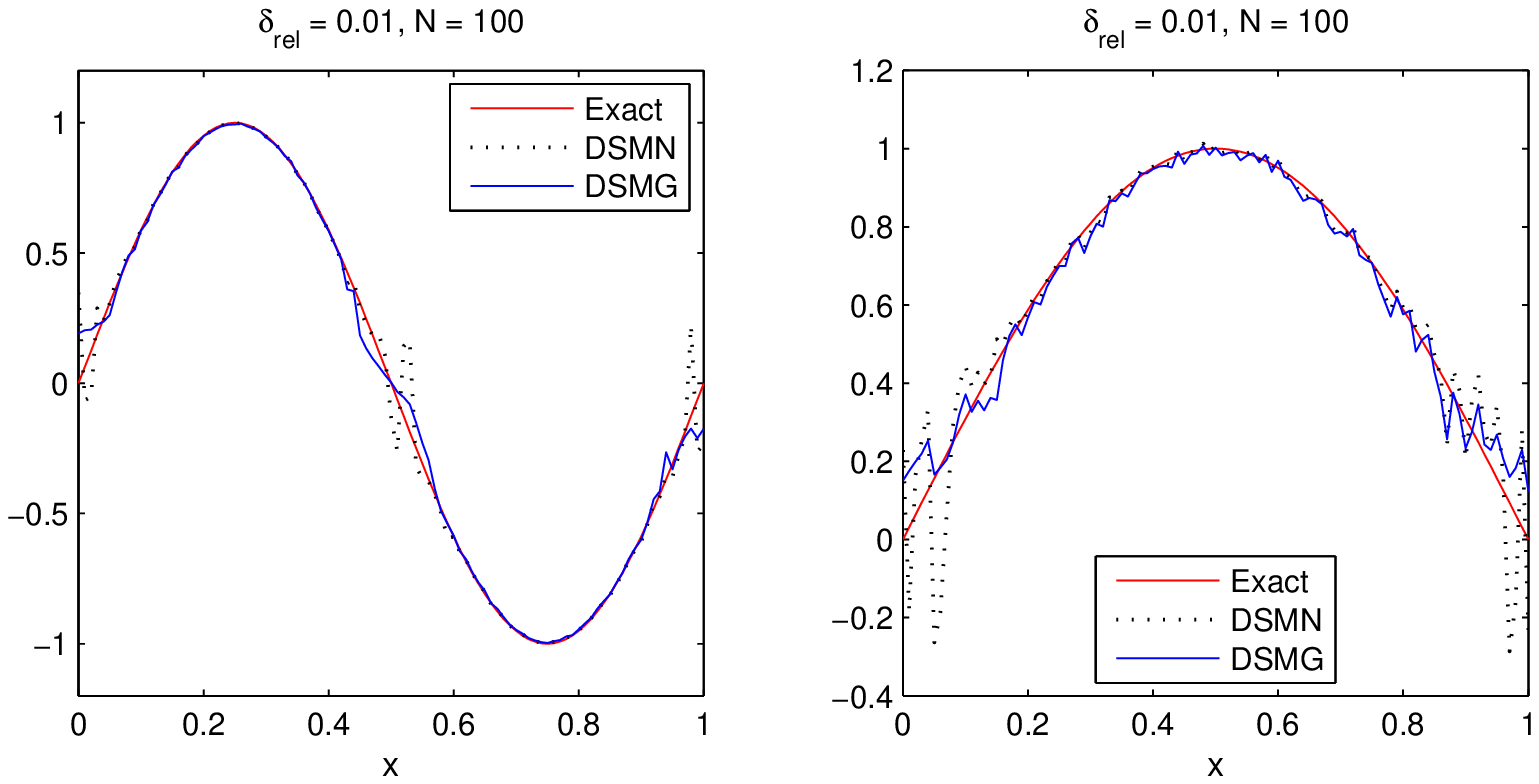}
}
\caption{Plots of solutions obtained by the DSMN and DSMG when $N = 100$, $u(x)=\sin(2\pi x)$, $x\in[0,1]$, $\delta_{rel}=0.01$ (left) 
and $N = 100$, $u(x)=\sin(\pi x)$, $x\in[0,1]$, $\delta_{rel}=0.01$ (right).}
\label{figmono2}
\end{figure}

Numerical experiments agree with the theory that the convergence rate of the DSMG is slower than that of the DSMN.
It is because the rate of decaying of the sequence $\{\frac{1}{(1+n)^\frac{1}{4}}\}_{n=1}^\infty$ 
is much slower than that of the sequence $\{\frac{1}{1+n}\}_{n=1}^\infty$. However, if the cost for evaluating
$F$ and $F'$ are not counted then the cost of computation at one iteration of the DSMG is of $O(N^2)$
 while that of the DSMN in one iteration of the DSMN is of $O(N^3)$.
 Here $N$ is the number of the nodal points. 
Thus, for large scale problems, the DSMG might be an alternative to the DSMN.
Also, as it is showed in Figure~\ref{figmono2}, the DSMG might yield solutions with better accuracy.

Experiments show that the DSMN still works with $a_n=\frac{a_0}{(1+n)^b}$ for $\frac{1}{4}\le b\le 1$.
So in practice, one might use faster decaying sequence $a_n$ to reduce the time of computation.

From the numerical results we conclude that the proposed stopping rule yields 
good results in this problem.

\end{document}